\documentclass[12pt]{amsart}
\usepackage{graphicx}
\usepackage[all]{xy}
\usepackage{amsmath,amssymb,amscd,amsfonts,mathrsfs}
\usepackage[mathcal]{euscript}
\usepackage[disable, colorinlistoftodos]{todonotes}

\usepackage{layout}
\usepackage{anysize}
\marginsize{22mm}{18mm}{17mm}{30mm}

\newtheorem{thm}{Theorem}[section]

\newtheorem{cor}[thm]{Corollary}
\newtheorem{lem}[thm]{Lemma}
\newtheorem{prop}[thm]{Proposition}
\theoremstyle{definition}

\theoremstyle{remark}
\newtheorem{rem}[thm]{Remark}
\numberwithin{equation}{section}

\newcommand{\symb}[2]{\left(
                \begin{matrix}
                 #1 \\
                 #2
            \end{matrix}\right)}
\newcommand{\set}[1]{\left\{#1\right\}}
\newcommand{\R}{\mathbb R}
\newcommand{\To}{\longrightarrow}
\newcommand{\Q}{\mathbb{Q}}
\newcommand{\Z}{\mathbb{Z}}
\newcommand{\C}{\mathbb{C}}
\newcommand{\N}{\mathbb{N}}
\newcommand{\Tr}{\text{Tr}}

\newcommand{\smat}[4]{\left(
                \begin{smallmatrix}
                 #1 & #2\\
                 #3 & #4
            \end{smallmatrix}\right)}
\newcommand{\mat}[4]{\left(
                       \begin{array}{cc}
                         #1 & #2 \\
                         #3 & #4 \\
                       \end{array}
                     \right)}
\newcommand{\h}{\mathfrak{h}}
\def\Imm{\operatorname{Im}}
\def\Rea{\operatorname{Re}}

\newcommand{\B}{\mathscr{B}}
\newcommand{\Rz}{\mathscr{R}_z}
\newcommand{\Xz}{X_z}
\newcommand{\rz}{R_z}
\newcommand{\Hz}{H_z}
\newcommand{\hz}{\mathcal{H}_z}
\newcommand{\Iz}{I_z}
\newcommand{\Dalg}{B}

\newcommand{\Oo}{\mathcal{O}}
\newcommand{\ror}{\mathcal{R}}

\newcommand{\emb}{\frac{v+1}{2}}
\newcommand{\Ups}{\Upsilon}
\newcommand{\Qq}{\mathcal{Q}}
\newcommand{\bQq}{\bar{\Qq}}
\newcommand{\thetaq}{\Theta_{\Qq}}

\newcommand{\psin}{\psi_{\mathcal{N}}}
\newcommand{\thetar}{\Theta_{[\A,R],\Nn}}
\newcommand{\Lor}{\mathcal{M}}
\newcommand{\End}{\textrm{End}}
\newcommand{\Hom}{\textrm{Hom}}
\newcommand{\col}{\::\:}
\newcommand{\id}{\mathbf{1}_2}
\newcommand{\sd}{\sqrt{D}}
\newcommand{\A}{\mathfrak{a}}
\newcommand{\Nn}{\mathcal{N}}
\newcommand{\SL}{SL_2(\Z)}

\newcommand{\g}{\gamma}
\newcommand{\eps}{\varepsilon}
\newcommand{\norm}{\mathbf{N}}
\newcommand{\Mat}{\textrm{Mat}}
\newcommand{\disc}{\textrm{disc}}
\newcommand{\er}{\mathcal{E}_R}
\newcommand{\hr}{\mathcal{H}_R}
\newcommand{\hh}{\mathcal{H}}
\newcommand{\rs}{\mathscr{R}}
\newcommand{\Sp}{Sp_4(\Z)}
\newcommand{\hthetar}{{\Theta}_{[\A,R],\Nn}}

\newcommand{\into}{\hookrightarrow}
\newcommand{\htheta}{\hat{\theta}}
\newcommand{\har}{h^\eps_{[\A, R]}(-N)}
\newcommand{\epsar}{\eps_{[\A, R]}}
\newcommand{\Tan}{\tau_{\A\bar{\Nn}}}

\begin{document}

\title[Central values of Hecke L-series]{Split-CM points and central values of Hecke L-series}%
\author{Kimberly Hopkins }%
\address{UT Austin, Department of Mathematics C1200, Austin, TX 78712 }%
\email{khopkins@math.utexas.edu }%

\subjclass{ }%

\date{\today}
\begin{abstract}
Split-CM points are  points of the moduli space $\h_2/\Sp$ corresponding to products $E\times E'$ of elliptic curves with the same complex multiplication. We prove that the number of split-CM points in a given class of $\h_2/\Sp$ is related to the coefficients of a weight $3/2$ modular form studied by Eichler. The main application of this result is a formula for the central value $L(\psin,1)$ of a certain Hecke $L$-series. The Hecke character $\psin$ is a twist of the canonical Hecke character $\psi$ for the elliptic $\Q$-curve $A$ studied by Gross, and formulas for $L(\psi,1)$ as well as generalizations were proven by Villegas and Zagier. The formulas for $L(\psin,1)$ are easily computable and numerical examples are given.
 \end{abstract}
\maketitle

\makeatletter 
\providecommand\@dotsep{5} 
\makeatother 
\section{Introduction }\label{S:intro}
 
Let $D<0$, $|D|$ prime be the discriminant of an imaginary quadratic field $K$ with ring of integers $\Oo_K$. Suppose  $N$ is a prime which splits in $\Oo_K$ and is divisible by an ideal $\Nn$ of norm $N$. We will define   Hecke characters $\psin$ of $K$ of weight one and conductor $\Nn$ (see Section \ref{S:results}).   These are twists of the canonical Hecke characters studied by Rohrlich  \cite{Ro1,Ro2,Ro3} and Shimura \cite{ShA,ShB,ShC}. Denote by $L(\psin,s)$ the corresponding Hecke $L$-series.

Our main theorem (Theorem \ref{T:mainthm}) is a formula in the spirit of  Waldspurger's results \cite{Wa1,Wa}. It says approximately that
\begin{equation}\label{E:intromainthm}
	L(\psin, 1) = \sum_{[R]} \sum_{[\A]}  \hthetar\cdot \har .
\end{equation}
Here the first sum is over all conjugacy classes of maximal orders $R$ in the quaternion algebra  ramified only at $\infty$ and $|D|$, and the second sum is over the elements $[\A]$ of the ideal class group of $\Oo_K$.  We will see that the $\har$ are integers related to coefficients of a certain weight $3/2$ modular form, and that the  $\hthetar$ are algebraic integers equal to the value of a symplectic theta function on `split-CM' points (defined in Section \ref{S:results}) in the Siegel space $\h_2/\Sp$. We expect the formula \eqref{E:intromainthm} to  be useful for computing the central value $L(\psin,1)$.

Let $A(|D|)$ denote a $\Q$-curve as defined in  \cite{Gro}.  This is an elliptic curve defined over the Hilbert class field $H$ of $K$ with complex multiplication by $\Oo_K$ which is isogenous over $H$ to its Galois conjugates.  Its $L$-series is a product of the squares of $L$-series $L(\psi, s)$ over the $h(D)$ Hecke characters of  conductor $(\sqrt{D})$. A formula for the central value $L(\psi, 1)$ expressed as a square of  linear combinations of certain theta functions was proven by  Villegas in \cite{V2}.  Extensions of his result to higher weight Hecke characters were given  by Villegas in \cite{V3} and jointly with Zagier in \cite{VZ}. The Hecke character $\psin$ is a twist of $\psi$ by a quadratic Dirichlet character of conductor $(\sqrt{D})\Nn$. Therefore our result \eqref{E:intromainthm} gives a formula for the central value of the corresponding twist of $A(|D|)$. 

Our main theorem can be stated in a particularly nice form when the  class number of $\Oo_K$ is one. Then  $[\A]=[\Nn]=[\Oo_K]$ and so in particular $\hthetar  = {\Theta}_{[R]}$ and $\har = h_R^\eps(-N)$ are independent of $[\A]$ and $\Nn$. This suggests that formula \eqref{E:intromainthm} will lead to a generating series for $L(\psin,1)$ as $N$ varies in terms of linear combinations (with scalars in $\{\Theta_{[R]}\}$) of half-integer weight modular forms.

We hope to extend these results to higher weight as follows.   For certain $k\in \Z_{\geq1}$ it is well-known that the central value $L(\psin^k, k)$ can be written as a trace over the class group of $\Oo_K$ of a weight $k$ Eisenstein series evaluated at Heegner points of level $N$ and discriminant $D$. It is a general philosophy (see \cite{Za}, for example) that such traces  relate to coefficients of a corresponding modular form of half-integer weight.  By the Siegel-Weil formula\footnote{The precise statement of this formula is simplified here for the sake of exposition.} we can write the central value of $L(\psin^k,s)$ 	in terms of a sum of  theta-series\footnote{Here $\omega_Q$ is the number of automorphisms of the form $Q$.}

\begin{equation}\label{E:higherweight}
	L(\psin^k, k) \overset{\cdot}=  \sum_{[\A]} \sum_{[Q]} \frac{1}{\omega_Q} \Theta_Q(\tau_{\A}).
\end{equation}
Here the sum is over $[\A]$ in the class  group of $\Oo_K$ and over classes of positive definite quadratic forms $Q:\Z^{2k}\To \Z$ in $2k$ variables and in a given genus. The point $\tau_\A \in \h$ is a Heegner point of level $N$ and discriminant $D$. Analogous to the case of two variables, these quadratic forms correspond  to higher rank Hermitian forms  (see \cite{Ot} and \cite{HK1,HK2,HI1,HI2,HI3}). An approach to counting the number of distinct theta values in \eqref{E:higherweight}  would be to associate the Hermitian forms to isomorphism classes of rank $k$ $R$-modules of $\Dalg$, for maximal orders $R$ of $\Dalg$. This paper does this   for the case $k=1$. Our intention here is to lay the groundwork  for the generalization to arbitrary weight $k$.

This paper is organized as follows. Basic notation is given in Section \ref{S:notation}. Background and a statement of results are in Section \ref{S:results}.  In Section \ref{S:heegel}, we analyze the endomorphisms of the principally polarized abelian varieties for the split-CM points, and show they form an explicit maximal order in the quaternion algebra $\Dalg$. In Section \ref{S:splitcm} we identify these orders with explicit right orders in $\Dalg$. In Section \ref{S:centralvalue} we prove the main results (Theorems \ref{T:firsthm}, \ref{T:secondthm} and \ref{T:mainthm}) and provide numerical  examples.  

\section{Notation}\label{S:notation}

Given any imaginary quadratic field $M$ of discriminant $d<0$, we denote by $\Oo_M$ its ring of integers, $Cl(\Oo_M)$ its ideal class group, $h(d)$ its  class number, and $Cl(d)$ the isomorphic class group of primitive positive definite binary quadratic forms of discriminant $d$. A nonzero integral ideal of $\Oo_M$ with no rational integral divisors besides $\pm 1$ is said to be \emph{primitive}. Any primitive  ideal $\A$ of $\Oo_M$ can be written uniquely as the $\Z$-module
\[
	\A = a\Z + \frac{-b+\sqrt{d}}{2}\: \Z =  [a, \frac{-b+\sqrt{d}}{2}]
\]
with $a:= \norm \A$ the norm of $\A$, and $b$ an integer defined modulo $2a$ which satisfies $b^2 \equiv d\bmod 4a$. Conversely any $a, b \in \Z$  which satisfy the conditions above determine a primitive ideal of $\Oo_M$. The coefficients of the corresponding primitive positive definite binary quadratic form are given by $[a,b, c:= \frac{b^2 - D}{4a}]$. The form $[a,-b,c]$ corresponds to the ideal $\bar{\A}$. We will always assume our forms are primitive positive definite and the same for  ideals. The point
\[
	\tau_\A : = \frac{-b + \sqrt{d}}{2a} 
\]
is in the upper half-plane $\h$ of $\C$ and is referred to in general as a \emph{CM point}. A \emph{Heegner point} of level $N$ and discriminant $D$ is a CM point $\tau_\A$ where $\A$ is given by  a form $[a,b,c]$ of discriminant $D$ such that $N|a$. The \emph{root} of $\tau_\A$  is defined to be the reduced representative $r\in (\Z/2N\Z)^\times$ such that $b\equiv r\bmod 2N$. 

Square brackets $[\cdot]$ around an object will denote its respective equivalence class. The units of a ring $R$ are written as $R^\times$.

\section{Statement of Results}\label{S:results}
We first  recall some basic results for Siegel space and symplectic modular forms. 

Assume $K$  is an imaginary quadratic field of prime discriminant $D<-4$. Let $L$ be an imaginary quadratic field of discriminant $-N<0$ where $N$ is a prime which splits in $\Oo_K$, and is divisible by an ideal $\Nn$ of norm $N$.     Note $h(D)$ and $h(-N)$ are both odd since $|D|$ and $N$ are prime. 
Let $\mu: \Oo_K/\Nn \To \Z/N\Z$ be the natural isomorphism. Composing this with the Jacobi symbol $(\frac{\cdot}{N}): \Z/N\Z \To \{0,\pm1\}$ defines a character
\[
	\chi: (\Oo_K/\Nn)^\times \To \{\pm1\}.
\]
This is an odd quadratic Dirichlet character of conductor $\Nn$. Let $I_\Nn$ denote the group of nonzero fractional ideals of $K$ which are coprime to $\Nn$, and let $P_\Nn\subset I_\Nn$ be the subgroup of principal ideals. The map $\psin: P_\Nn \To K^\times$ defined by
\[
	\psin((\alpha)):= \chi(\alpha)\alpha 
\]
is a homomorphism. There are exactly $h(D)$ extensions of $\psin$ to a Hecke character $\psin: I_\Nn \To \C^\times$. This produces $h(D)$ primitive Hecke characters of weight one and conductor $\Nn$. (See \cite{Gr,Pa} and \cite[p.225]{Ro1}  for more details). Fix a choice of $\psin$. We can extend $\psin$ to a multiplicative function on all of $\Oo_K$ by setting $\psin(\A):= 0$ if $\A$ is not coprime to $\Nn$.   

 To $\psin$ we associate the Hecke $L$-function
\[
	L(\psin, s) := \sum_{\A \subset \Oo_K} \frac{\psin(\A)}{\norm \A^s}, \qquad \Rea(s)>3/2.
\]

We now recall a result due to Hecke which gives the  central value $L(\psin, 1)$ as  a linear combination of certain theta series evaluated at CM points. For each primitive ideal $\Qq$  of $\Oo_L$, the associated theta series  is defined by
\[
	\thetaq(\tau):= \sum_{\lambda \in \Qq} q^{\norm(\lambda)/\norm(\Qq)}, \qquad q=e^{2\pi i \tau}, \:\:\tau \in \h.
\]
 It is a modular form on $\Gamma_0(N)$ of weight one and character $\textrm{sgn}(\cdot)\big(\frac{-N}{|\,\cdot\,|}\big)$ (see \cite[p.49]{Ei1}, for example). 
 
 For each primitive ideal $\A$ of $\Oo_K$ with norm prime to $N$, the product ideal $\A\bar{\Nn}$ is of the form $[a_1N, \frac{-b_1+\sqrt{D}}{2}]$ for some $a_1, b_1 \in \Z$. The point
\[
	\tau_{\A\bar{\Nn}}:= \frac{-b_1+\sqrt{D}}{2a_1N} \in \h
\]
is  a Heegner point of level $N$ and discriminant $D$. We will write $\tau_\A$ or just $\tau$ for $\tau_{\A\bar{\Nn}}$ when the context is clear. Note that as $\A$ runs over a distinct set of representatives of $Cl(\Oo_K)$, so does $\A\bar{\Nn}$. (The fact that representatives of $Cl(\Oo_K)$ can be chosen with norm prime to $N$ is in \cite[Lemmas 2.3, 2.25]{Co}, for example.) By  $\A$ we will always mean a primitive ideal  with norm prime to $\Nn$ as above. 

 Hecke's formula \cite{He} for the central value of $L(\psin, s)$ states
\begin{equation}\label{E:Lintheta}
	L(\psin, 1) = \frac{2\pi}{\omega_N \sqrt{N}} \sum_{[\A]\in Cl(\Oo_K)} \sum_{[\Qq] \in Cl(\Oo_L)} \frac{\thetaq(\tau_{\A\bar{\Nn}})}{\psi_{\bar{\Nn}}(\bar{\A})}
\end{equation}
where $\omega_N$ is the number of units in $\Oo_L$.  

The theta function for $\Qq$ arises from a certain specialization of a symplectic theta function. 
Let $Sp_4(\Z)$ denote the Siegel modular group of degree $2$. 
Let $\Gamma_\theta$ be the subgroup of $\smat{\alpha}{\beta}{\g}{\delta} \in Sp_4(\Z)$ ($\alpha, \beta, \g, \delta \in \Mat_{2}(\Z)$) such that both $\alpha \:{}^T\gamma$ and $\beta\:{}^T\delta$ have even diagonal entries. The group $\Gamma_\theta$ inherits the action of $Sp_4(\Z)$ on the Siegel upper half plane $\h_2:=\set{z \in \Mat_{2}(\C) \col {}^Tz=z, \Imm(z)>0}$. 	
 Define the symplectic theta function by
\[
	\theta(z):= \sum_{\vec{x}\in \Z^2} \exp[\:\pi i \:{}^T\vec{x}\,z\,\vec{x} \:], \qquad z\in \h_2.
\]
 The function  $\theta$ satisfies the functional equation
\begin{equation}\label{E:thetafcnal}
	\theta(M\circ z) = \chi(M) [\det(\gamma z+ \delta)]^{1/2} \theta(z), \qquad M \in \Gamma_\theta
\end{equation}
where $\chi(M)$ is an eighth root of unity which depends on the chosen square root of $\det(\gamma z +\delta)$ but is otherwise independent of $z$. It  is a symplectic modular form on $\Gamma_\theta$ of dimension $-1/2$ with multiplier system   $\chi$ (see  \cite[p.43]{Ei1} or \cite[p.189]{Mu}, for example)\footnote{The symplectic theta function is sometimes defined with extra parameters, $\theta(z, u, v)$ where $u, v\in \C^2$, in which case the theta function above is equal to $\theta(z, \vec{0}, \vec{0})$.}. 

 Given a primitive ideal $\Qq$ of $\Oo_L$, let 
 $Q:=[a,b,c]$ represent the corresponding  binary quadratic form of discriminant $-N$. The  product of the matrix of $Q$ with any Heegner point $\tau_\A$ is the Siegel point
\[
	Q\tau_\A := \mat{2a}{b}{b}{2c} \cdot \tau_\A \in \h_2.
\]
We will refer to points constructed in this way as \emph{split-CM points} of level $N$ and discriminant $D$.  This yields the relation
\begin{equation}\label{E:thetaid}
	\thetaq(\tau_\A) = \theta(Q\tau_\A)
\end{equation}
which can be substituted into  formula \eqref{E:Lintheta} to get
\begin{equation}\label{E:Lintheta2}
	L(\psin, 1) = \frac{2\pi}{\omega_N \sqrt{N}} \sum_{[\A]\in Cl(\Oo_K)} \sum_{[Q] \in Cl(-N)} \frac{\theta(Q\tau_\A)}{\psi_{\bar{\Nn}}(\bar{\A})}.
\end{equation}
If $Q\sim Q'$ in $Cl(-N)$, then $Q\tau_\A \sim Q'\tau_\A$ in $\h_2/\Sp$,  and if $\A\sim\A'$ in $Cl(\Oo_K)$, then   $Q\tau_{\A} \sim Q\tau_{\A'}$ in $\h_2/\Sp$  (see Remark \ref{R:equivQ} and Lemma \ref{L:RindA}). In addition it is shown in \cite[Lemma 53]{Pa} that these equivalences of Siegel points sustain modulo $\Gamma_\theta$. The function $\theta/\psi_{\bar{\Nn}}$ is invariant on such points:
\begin{lem}\label{L:thwelldef}
Fix an ideal $\A \subset \Oo_K$ and a prime ideal $\Nn\subset \Oo_K$ of norm $N$. Let $Q$ be a binary quadratic form of discriminant $-N$.   Then the value
\begin{equation}\label{E:thetaval}
	\frac{ \theta(Q\tau_{\A\bar{\Nn}})}{\psi_{\bar{\Nn}}(\bar{\A})}
\end{equation}
depends only on the class $[Q] \in Cl(\Oo_L)$  and the class $[\A] \in Cl(\Oo_K)$.	
\end{lem}
\begin{proof}
	The value $\theta(Q\tau_{\A\bar{\Nn}})$ is independent of the class representative of $[Q]$  because  equivalent forms represent the same values. That \eqref{E:thetaval}  is independent of the representative of $[\A]\in Cl(\Oo_K)$ is a short calculation using the functional equation for $\theta$ in \eqref{E:thetafcnal} and is done in \cite[Proposition 22]{Pa}. \end{proof}

Therefore the  set of points $[Q]\tau_{[\A]\bar{\Nn}}$ as $[Q]$ runs over $Cl(-N)$ and $[\A]$ runs over $Cl(\Oo_K)$ are equivalent in $\h_2/\Gamma_\theta$  and are identified under $\theta/\psi_{\bar{\Nn}}$. We  refer to $[Q]\tau_{[\A]\bar{\Nn}}$ as a \emph{split-CM orbit}. Thus to determine which  values $\theta(Q\tau_\A)$ are equal in \eqref{E:Lintheta2} it is necessary to determine which split-CM orbits $[Q]\tau_{[\A]\bar{\Nn}}$  are equivalent modulo $\Gamma_\theta$.   Since $\h_2/Sp_4(\Z)$ is a moduli space for the principally polarized abelian varieties  of dimension two  (\cite{Mu} or \cite[Chp. 8]{BL}), the classes of split-CM points  are determined by the isomorphism classes of the corresponding varieties. 


To describe these, we will recall  some basic facts about quaternion algebras.  Let $B:=(-1,D)_\Q$ be the quaternion algebra over $\Q$ ramified at $\infty$ and $|D|$. Recall two maximal orders $R$, $R'$ in $B$ are \emph{equivalent} if there exists $x\in \Dalg^\times$ such that $R' = x^{-1} R x$. Moreover, two optimal embeddings $\phi:\Oo_L \hookrightarrow R$ and $\phi' : \Oo_L \hookrightarrow R'$ are \emph{equivalent} if there exists $x \in \Dalg^\times$ and $r\in R'^\times$ such that $R' = x^{-1} R x$ and $\phi' = (xr)^{-1} \phi (xr)$. Let $\ror$ denote the set of conjugacy classes of maximal orders in $B$ and let $\Phi_\ror$ denote the set of classes of optimal embeddings of $\Oo_L$ into the maximal orders of $B$. Let $\ror_N\subset \ror$ denote the maximal order classes which admit an optimal embedding of $\Oo_L$.      Given an optimal embedding $(\phi:\Oo_L\into R )\in \Phi_\ror$, let $(\bar{\phi}:\Oo_L \into R) \in \Phi_\ror$ denote its quaternionic conjugate, so that $\phi(\sqrt{-N})= \bar{\phi}(-\sqrt{-N})$. The quotient $\Phi_\ror/-$ will denote the set $\Phi_\ror$ modulo this conjugation. Let $h_R(-N)$ denote the number of optimal embeddings  of $\Oo_L$ into $R$ modulo conjugation by $R^\times$. This number is an invariant of the choice of representative of $[R]$ in $\ror$. 

Our first theorem says that the classes of split-CM points in Siegel space correspond to classes of maximal orders in $B$.
\begin{thm}\label{T:firsthm}
 Fix $[\A] \in Cl(\Oo_K)$, $\Nn \subset \Oo_K$ a prime ideal of norm $N$, and $\tau:= \Tan$. There is a bijection  
 \[
 	\Ups_1 : \set{ Q\tau \col [Q] \in Cl(-N)}/\Sp \To \ror_N.
 \] 
 This map is independent of the choice of representative $\A$ of $[\A]$. 
\end{thm}

Let $\Ups_1^{-1}([R])$ for $[R]\in \ror_N$ denote the pre-image class in $\h_2/\Sp$ and set  $\Ups_1^{-1}([R]):= \emptyset$ if $[R]\in \ror\setminus\ror_N$. Our second theorem gives the number of split-CM orbits in a given class.
\begin{thm}\label{T:secondthm}
Assume the hypotheses of Theorem \ref{T:firsthm}.  For any $[R] \in \ror$, 
\[
	\#\set{ [Q]\tau  \in \Ups_1^{-1}([R]): [Q] \in Cl(-N)}  = h_R(-N)/2. 
\]
That is, the number of split-CM orbits in the class in $\h_2/\Sp$ corresponding to $[R]$ under Theorem \ref{T:firsthm} is $h_R(-N)/2$. 
\end{thm}

 For a maximal order $R$ of $B$, define  $S_R:= \Z + 2R$ and $S_R^0\subset S_R$ to be the suborder of trace zero elements. The suborder $S_R^0$ is a rank $3$ $\Z$-submodule of $R$. Define $g_R$ to be its theta series
\begin{align*}
	g_R(\tau)&:= \frac12 \sum_{x\in S_R^0} q^{\N(x)} \\
		&= \frac12 + \sum_{N>0} a_R(N)q^N,
\end{align*}
where $a_R(N)$ are defined by its $q$-expansion. 
It is well known that $g_R$ is a  weight $3/2$ modular form on $\Gamma_0(4|D|)$.  Applying \cite[Proposition $12.9$]{Gro2}  to fundamental $-N$ gives
\[
	a_R(N) = \frac{\omega_R}{\omega_N} h_R(-N)
\]
where $\omega_R$ is the cardinality of the set $R^\times/<\pm 1>$. 

This gives immediately the following Corollary to Theorem \ref{T:secondthm}.
\begin{cor}
Assume the hypotheses of Theorem \ref{T:secondthm}. For any $[R]\in \ror$, 
\[
	\#\set{ [Q]\tau  \in \Ups_1^{-1}([R]): [Q] \in Cl(-N)}  = a_R(N)\cdot \frac{2\omega_N}{\omega_R}. 
\]
That is, the number of split-CM orbits in the class in $\h_2/\Sp$ corresponding to $[R]$ under Theorem \ref{T:firsthm} is proportional to the $N$-th Fourier coefficient of the weight $3/2$ modular form $g_R$. 
\end{cor}

The application of Theorems \ref{T:firsthm} and \ref{T:secondthm} to a formula for $L(\psin, 1)$ proceeds as follows.  
Define the following normalization of $\theta$  given by \cite{Pa}:
\begin{equation}\label{E:thetanorm}
\htheta(Q\tau_{\A\bar{\Nn}}) := \frac{\theta(Q\tau_{\A\bar{\Nn}})}{\eta(\bar{\Nn})\eta(\Oo_K)}
\end{equation}
where $\eta(z):= e_{24}(z) \prod_{n=1}^\infty (1- e^{2\pi iz})$ for $\Imm(z)>0$ is Dedekind's eta function and the evaluation of $\eta$ on ideals is defined in Section \ref{S:centralvalue}. It is proven in \cite[Proposition 23]{Pa} (see also \cite{HV}) that the numbers in $\htheta(Q\tau_{\A\bar{\Nn}})/\psi_{\bar{\Nn}}(\bar{\A})$ are algebraic integers.

Define 
\[
	\Theta_{[\A,  Q], \Nn}:= \frac{ \htheta(Q\tau_{\A\bar{\Nn}})}{\psi_{\bar{\Nn}}(\bar{\A})}.
\] 
This is well-defined by Lemma \ref{L:thwelldef}. The following lemma says that the  theta-values which correspond to a given class $[R]\in \ror$ under Theorem \ref{T:firsthm}  are all equal up to $\pm1$. 
	
\begin{lem}\label{L:pmdiff}
 Fix $[\A] \in Cl(\Oo_K)$, $\Nn \subset \Oo_K$ a prime ideal of norm $N$, and $\tau:= \Tan$. Let $[R]\in \ror$. 
Then the values 
\begin{equation}\label{E:thetagp}
	\set{ \Theta_{[\A, Q], \Nn} \col  [Q]\tau\in \Ups_1^{-1}([R]) }
\end{equation}
differ by $\pm1$.  
\end{lem}

Assume Lemma \ref{L:pmdiff} holds (see Section \ref{S:splitcm} for the proof). Given $[R]\in \ror_N$ and any $[Q]\tau \in \Ups_1^{-1}([R])$, define $\hthetar$ to be either $\Theta_{[\A,  Q], \Nn}$ or $-\Theta_{[\A,  Q], \Nn}$ so that it satisfies $\text{Re}(\hthetar)>0$. Set $\hthetar:=0$ if $[R]\in \ror\setminus \ror_N$.

We record the mysterious $\pm1$ signs appearing in Lemma \ref{L:pmdiff} by defining 
\begin{align}\label{E:opsign}
	\epsar : \set{[Q]\tau\in \Ups_1^{-1}([R])} &\To \set{\pm1}\\\nonumber
		[Q]\tau &\mapsto \text{sgn}\left(\text{Re}\left(\Theta_{[\A,Q], \Nn}\right)\right).
\end{align}
Note $\Theta_{[\A, Q], \Nn}= \pm \hthetar$ by construction. This definition assigns, albeit somewhat arbitrarily, a fixed choice of sign for the theta-values as $[Q]$ varies.

 We then define a corresponding twisted variant of $h_R(-N)$ by
\begin{equation}\label{E:har}
	h_{[\A, R]}^\eps(-N):= \sum_{[Q]\tau\in \Ups_1^{-1}([R])} \epsar([Q]\tau).
\end{equation}
The formula for $L(\psin,1)$ can now be stated as follows.

\begin{thm}\label{T:mainthm} Let $\Nn \subset \Oo_K$ be a prime ideal of norm $N$. Then 
\begin{equation}\label{E:main}
	L(\psin, 1) = \frac{\pi\cdot\eta(\bar{\Nn})\eta(\Oo_K)}{\omega_N \sqrt{N}}\sum_{[R]\in \ror} \sum_{[\A] \in Cl(\Oo_K) } \hthetar \cdot \har.
\end{equation}
where $\hthetar$ is an algebraic integer and $\har$ is an integer with $|\har|\leq h_R(-N)$. 
\end{thm}

\begin{rem}
The signs in Lemma \ref{L:pmdiff} and hence the function $\har$ depend on the character $\chi$  which appears in the functional equation \eqref{E:thfcnl} for $\theta$. In particular, the values of $\chi$ depend on the entries of the transformation matrices in $\Gamma_\theta$ which takes one Siegel point to an equivalent one. This value  is complicated to compute or even define, and is discussed in detail in \cite{AnMa,St} and \cite[Appendix to Chp 1]{Ei1}.  An arithmetic formula for these signs  and for $\har$ is yet to be determined. But  since  the $\har$ are a weighted count of optimal embeddings, we expect that, like the $h_R(-N)$, they will be related to coefficients of a half-integer weight modular form.   This will be treated in a subsequent paper.  
\end{rem}

Theorem \ref{T:mainthm} gives us an upper bound on $L(\psin,1)$ in terms of the computable modular form coefficients $h_R(-N)$.
\begin{cor}\label{C:wald}
	Assume the hypotheses of Theorem \ref{T:mainthm}. Then 
\[
	|L(\psin, 1)| \leq \frac{\pi\cdot|\eta(\bar{\Nn})\eta(\Oo_K)|}{\omega_N \sqrt{N}}\sum_{[R]\in \ror}\sum_{[\A] \in Cl(\Oo_K) }  |\thetar| \cdot h_R(-N).
\]
\end{cor}


If $h(D)=1$,  then  \eqref{E:main} has a particularly  simple form:
\begin{cor}  Assume the hypotheses of Theorem \ref{T:mainthm} and suppose $h(D)=1$. Then $\hthetar = {\Theta}_{[R]}$ and $\har=h^\eps_{[R]}(-N)$ are independent of $\A$ and $\Nn$ and
\[
	L(\psin, 1) = \frac{\pi\cdot |\eta(\Oo_K)|^2}{\omega_N \sqrt{N}} \sum_{[R]\in \ror} \Theta_{[R]} \cdot h^\eps_{[R]}(-N).
\]

\end{cor}

We conclude this section with a comment regarding varying $N$. The set
\[
	\bigcup_{N} \big\{ [Q]\tau_{[\A]\bar{\Nn}} \col [Q]\in Cl(-N),\quad [\A]\in Cl(\Oo_K),\quad \Nn\subset\Oo_K \text{ of norm } N\big\},
\]
 of split-CM orbits over all prime $N$ with $D\equiv\Box\bmod 4N$ partitions into a  finite number of Siegel classes in $\h_2/\Sp$. This has a natural explanation from our viewpoint.  As a complex torus,  $X_{Q\tau}$ is isomorphic  to a product $E\times E'$ of two elliptic curves $E, E'$ defined over $\bar{\Q}$ and with complex multiplication by $\Oo_K$. (This is the reason the $Q\tau$ are called `split-CM'.) It is a general result of \cite{NN} that there are only finitely many principal polarizations on a given complex abelian variety up to isomorphism. There are also only finitely many isomorphism classes of elliptic curves with CM by $\Oo_K$. Together these imply that the number of classes of Siegel points $(X_{Q\tau}, H_{Q\tau})$ for all split-CM points $Q\tau$ of discriminant $D$ must be finite. See \cite[Theorem 58]{Pa} as well for an alternative interpretation.

\section{Endomorphisms of $\Xz$ preserving $\Hz$  }\label{S:heegel}

In this section we prove that the endomorphisms of the abelian varieties corresponding to split-CM points give maximal orders in the quaternion algebra $\Dalg=(-1,D)_\Q$. Let $V, V'$ be complex vector spaces of dimension $2$ with lattices $L\subset V$, $L'\subset V'$.  The analytic and rational representations are denoted by  $\rho_a: \Hom(X,X')\To \Hom_\C(V,V')$ and $\rho_r:\Hom(X,X') \To \Hom_\Z (L,L')$, respectively. Recall the periods matrices $\Pi, \Pi' \in \Mat_{2\times4}(\C)$ of $X, X'$  commute with $\rho_a$ and $\rho_r$ in the following diagram 
%
\begin{equation}\label{E:homdiag}
\xymatrix{ 
\Z^{2g} \ar[d]_{\rho_r(f)} \ar[r]^{\Pi} & \C^g \ar[d]^{\rho_a(f)}\\ 
\Z^{2g'} \ar[r]_{\Pi'} & \C^{g'}}
\end{equation}
(see \cite{BL}, for example).

For any Siegel point $z\in \h_2$, let $\Pi_z:= [z, \id] \in \Mat_{2\times 4}(\C)$  be its period matrix,  $L_z:=\Pi_z \Z^4$ be its defining lattice, and $X_z:= \C^2/L_z$ be its corresponding complex torus. The Hermitian form $\hz:\C^2\times\C^2\to \C$ defined by $\hz(u,v):= \:{}^Tu \Imm(z)^{-1} \bar{v}$ determines a principal polarization on $X_z$. As a point in the moduli space $\h_2/Sp_4(\Z)$, $z$ corresponds to the principally polarized abelian variety $(X_z, \hz)$. Throughout Sections \ref{S:heegel}, \ref{S:splitcm} and \ref{S:centralvalue}, fix a representative $\A$ of $[\A] \in Cl(\Oo_K)$, $\Nn \subset \Oo_K$ a prime ideal of norm $N$, $\tau:= \Tan:=\frac{-b_1+\sqrt{D}}{2a_1N}$, and a split-CM point $z=Q\tau$ of level $N$ and discriminant $D$ where  $Q:=[a,b,c]$ is of discriminant $-N$. The endomorphisms of $(X_z, \hz)$ will be our first main object of study. 

We define $\B$ to be the $\Q$-algebra of endomorphisms of $X_z$ which fix $\hz$
\[
	\B := \set{ \alpha \in \End_\Q(\Xz)\col \hz(\alpha u, v) = \hz(u, \alpha^\iota v)\quad\forall\: u,v \in \C^2};
\]
here $\iota$ is the canonical involution inherited from $\Mat_2(K)$ as defined in \cite{Sh}. In terms of matrices, let $\Hz:=\Imm(z)^{-1}$ denote the matrix of $\hz$ with respect to the standard basis of $\C^2$. Then viewing $\End_\Q(\Xz) \subseteq \Mat_{2}(K)$,  the set $\B$ is
\[
	\B = \set{ M \in \End_\Q(\Xz)\col {}^T\!\bar{M}\Hz = \Hz M^\iota }.
\]
The bar denotes complex conjugation restricted to $K$. The map $\iota$ sends a matrix $M$ to its adjoint, or equivalently sends $M$ to $\Tr(M)\cdot \id-M$. 

We define $\Rz$ to be the $\Z$-submodule  of endomorphisms which fix $\Hz$
\begin{equation}\label{E:defrz}
	\Rz := \set{M \in \End(\Xz )  \col {}^T\!\bar{M}\Hz = \Hz M^\iota }.
\end{equation}


 The first observation is that $\B$ is isomorphic to a rational definite quaternion algebra.
\begin{prop}\label{P:Bisquat}
$\B$  is isomorphic to $\Dalg$ as $\Q$-algebras. 
\end{prop}
\begin{rem}
In \cite[Proposition $2.6$]{Sh}, Shimura proves $\B$ is a quaternion algebra over $\Q$ in a much more general setting by showing $\B\otimes \bar{\Q}$ is isomorphic to $\Mat_2(\bar{\Q})$. Here we give an alternative proof which explicitly gives the primes ramified in $\B$. 
\end{rem}

\begin{proof}
We will need the following elementary lemma.
\begin{lem}\label{L:isoB}
 Suppose $Q_1, Q_2 \in \Mat_2(\Z) $ with determinant $N$. Set $H_i:= \Imm(Q_i\tau)^{-1}$ and 
\[
\mathscr{R}_i:= \set{ M \in \End(X_{Q_i\tau}) \col {}^T\!\bar{M}H_i = H_i M^\iota}, \qquad i=1,2.
\]  Let $S=\Z$ or $\Q$ and suppose there exists $A \in GL_2(S)$  such that $Q_2 = (\det A)^{-1} A Q_1 {}^T\! A$. Then the map 
\begin{align}
	\End_S(X_{Q_1\tau})  &\To \End_S(X_{Q_2\tau})\\\nonumber
	M &\mapsto AMA^{-1}
\end{align}
and the induced map
\[
	\mathscr{R}_1\otimes_\Z S \To \mathscr{R}_2\otimes_\Z S
\]  
are $S$-algebra isomorphisms. 

\end{lem}
\begin{proof}[Proof of Lemma]
Let $\Pi_i:= [Q_i\tau, \id]$ be the period matrices for  $Q_i\tau$, $i=1,2$. Suppose $M \in \End_S(X_{Q_1\tau})$. By \eqref{E:homdiag}, this is if and only if $M\Pi_i = \Pi_i P$ for some $P \in \Mat_4(S)$.
 Set
  \[
\tilde{A}:= \mat{(\det A^{-1}){}^T\!A}{0}{0}{A^{-1}} \in GL_4(S).
\]  Using the  identity $A \Pi_1 \tilde{A} = \Pi_2$ gives
\[
	(AMA^{-1}) \Pi_2 = \Pi_2 (\tilde{A}^{-1} P \tilde{A}).
\]
Clearly $\tilde{A}^{-1} P \tilde{A} \in \Mat_4(S)$, hence $AMA^{-1} \in \End_S(X_{Q_2\tau})$. 

Furthermore the identity $H_1 = (\det A^{-1}) \: {}^T\! A H_2 A$ implies ${}^T(\overline{AMA^{-1}}) H_2 = H_2 (AMA^{-1})^\iota$ by a straightforward calculation. 
\end{proof}

Define matrices 
\begin{equation}\label{E:defA}
A:= \frac{1}{2a}\mat{1}{0}{-b}{2a} \in GL_2(\Q)\quad \text{ and } \quad Q':= \mat{1}{0}{0}{N}.
\end{equation}
 By Lemma \ref{L:isoB}, $\B$ is isomorphic as a $\Q$-algebra to 
\[
	\B' := \set{ M \in \End_\Q(X_{Q'\tau})\col {}^T\!\bar{M}H' = H' M^\iota }
\]
where $H':= \Imm(Q'\tau)^{-1} $. 


We will compute $\B'$ explicitly. Let $E_\tau := \C/(\Z+\Z \tau)$ for any $\tau \in \h$. Clearly $X_{Q'\tau}\cong E_\tau \times E_{N\tau}$ as complex tori. The endomorphisms of $X_{Q'\tau}$ are characterized as follows.
\begin{lem}\label{L:end}
\[
	\End(X_{Q'\tau}) =  \mat{\Oo_K}{\Z + \Z\omega/N}{N\Z + \Z\bar{\omega}}{\Oo_K}
\]
where $\omega:=a_1N\tau$.
\end{lem}
Assuming this for a moment, we have $\End_\Q(X_{Q'\tau}) = \Mat_2(K)$, and a quick calculation shows any $M = \smat{\alpha}{\beta}{\gamma}{\delta} \in \Mat_2(K)$ satisfies ${}^T\!\bar{M} H = H M^\iota$ if and only if $\delta = \bar{\alpha}$ and $\gamma = -N\bar{\beta}$. Therefore
\[
	\B' = \left\{ \mat{\alpha}{\beta}{-N\bar{\beta}}{\bar{\alpha}} \col \alpha, \beta \in K\right\}\subset \Mat_2(K).
\]

The elements
\[
	\mat{1}{0}{0}{1}, \mat{\sd}{0}{0}{-\sd}, \mat{0}{1}{-N}{0}, \mat{0}{\sd}{N\sd}{0}
\]
form a basis of $\B'$ and clearly give an isomorphism to $(D,-N)_\Q$. We claim $\Dalg\cong (D,-N)_\Q$. This is a general fact: if $p$, $q$ are primes with $p\equiv q\equiv 3\bmod 4$ and $-p$ is a square modulo $q$, then $(-p,-q)_\Q$ is ramified at $\infty$ and $p$ only, so $(-p,-q)_\Q\cong (-1,p)_\Q$.  Hence $\B\cong \B'\cong \Dalg$ as $\Q$-algebras. 

It remains to prove Lemma \ref{L:end}.

\begin{proof}[Proof of Lemma \ref{L:end}]
	 For any quadratic surds $\tau, \tau' \in K$, 
	\[
		\Hom(E_\tau, E_{\tau'}) = \set{\alpha \in K: \alpha(\Z+\Z\tau)\subseteq \Z+\Z\tau'}.
	\]
	
Since $X_{Q'\tau} \cong E_\tau \times E_{N\tau}$, we have 
\[
	\End(X_{Q'\tau}) = \mat{\End(E_\tau)}{\Hom(E_{N\tau}, E_\tau)}{\Hom(E_\tau, E_{N\tau})}{\End(E_{N\tau})}.
\]

We compute. $\End(E_{N\tau}) = \Oo_K$ since $\Z+\Z a_1N\tau = \Oo_K$ and  $[1, N\tau]$ is a (proper) fractional $\Oo_K$-ideal. Similarly
$\End(E_\tau) = \Oo_K$ since $\Z+\Z\tau$ is a fractional $\Oo_K$-ideal. 

It is straightforward to check $\Z+\Z a_1\tau \subseteq \Hom(E_{N\tau},E_\tau)$. On the other hand, $\Hom(E_{N\tau}, E_\tau) \subset \Z+\Z\tau$ by definition, and this is proper containment since otherwise $\Z+\Z N\tau$ would preserve $\Z+\Z\tau$ which is impossible since the former contains $\Oo_K$. Therefore $\Hom(E_{N\tau},E_\tau) = \Z+\Z m\tau$ for some integer $m|a_1$ but a quick calculation shows $m=a_1$ else it divides $a_1, b_1$ and $c_1$ whose gcd is assumed to be $1$.  

It remains to show
\[
	\Hom(E_\tau, E_{N\tau}) = N\Z + \Z \bar{\omega}.
\]
First observe the ideal $(N)$ in $\Oo_K$ is contained in $\Hom(E_\tau, E_{N\tau})$ since
\[
	N(\Z + \Z a_1 N\tau) (\Z+ \Z\tau) \subseteq N(\Z+\Z\tau) \subseteq \Z + N\Z\tau.
\]
Furthermore $(N)$ splits as $(N) = \Nn\cdot \bar{\Nn}$ where $\Nn= N\Z + \Z\omega$. Therefore 
\[
	\Nn\cdot \bar{\Nn} \subseteq \Hom(E_\tau, E_{N\tau}) \subseteq \Oo_K,
\]
where the last containment follows because $\Z+\Z\tau$ is a proper fractional $\Oo_K$-ideal  which contains $\Z+\Z N\tau$. But since $\Oo_K$ is Noetherian, there exists a maximal order $M$ such that
\[
	\Nn\cdot \bar{\Nn} \subseteq \Hom(E_\tau, E_{N\tau}) \subseteq M \subseteq \Oo_K.
\]
Therefore either $\Nn$ or $\bar{\Nn}$ is in $M$. Whichever is contained in $M$ is actually equal to $M$ since they are both prime and hence maximal. But $\Hom(E_\tau, E_{N\tau})$ is not contained in $\Nn$. For example, $\bar{\omega} \in \Hom(E_\tau, E_{N\tau})$ but not in $\Nn$. Thus
\[
	\Hom(E_\tau, E_{N\tau}) \subseteq \bar{\Nn}.
\]

Finally since the index $[\bar{\Nn}: (N)]=N$ is prime, either $\Hom(E_\tau, E_{N\tau}) $ is equal to $\Nn$ or  $\bar{\Nn}$, but we already showed the former is impossible, hence it is the latter. 
\end{proof}
This also completes the proof of Proposition \ref{P:Bisquat}.
 \end{proof}

\begin{lem}\label{L:Rzorder}
$\Rz$ is isomorphic to an order in $\Dalg$ as $\Z$-algebras, and admits an optimal embedding of $\Oo_L$.
\end{lem}
\begin{proof}
The first part is immediate. 

The embedding is given in matrix form by $QS$ where $S:=\smat{0}{1}{-1}{0}$. It is straightforward to check that $(QS)^2=-N$ and $\frac{1+QS}{2} \in \Rz$ using definition \eqref{E:defrz}. An embedding is optimal if it does not extend to any larger order in the quotient field, but this is immediate since $\Oo_L$ is the maximal order in $L$. (See \cite{Sh} for additional discussion of this order.) 
\end{proof}

The next step is to prove the order $\Rz$ is maximal. 

\begin{thm}\label{T:Rzismax}
$\Rz$ is a maximal order.
\end{thm}
\begin{proof}

It suffices to show the local order $(\Rz)_p$ is maximal for all primes $p$. We do this with the following  two lemmas.

\begin{lem}$(\Rz)_p$ is maximal for all primes $p\neq2$.
\end{lem}
\begin{proof}[Proof of Lemma]
Define $\rs':=\B'\cap\End(Q'\tau)$ with $Q'$ defined in \eqref{E:defA}. From Lemma \ref{L:end} and the definition of $\B'$ above it is clear that $\rs'$ is an order  given explicitly by 
\begin{equation}\label{E:Rprime}
	\rs'= \set{ \mat{\alpha}{\beta}{-N\bar{\beta}}{\bar{\alpha}} \col \alpha \in \Oo_K, \beta\in\Z+\Z\omega/N}.
\end{equation}
Its discriminant is $D^2$, which can be computed using the basis 
\begin{equation}\label{E:uibasis}
	u_1:= \mat{1}{0}{0}{1},u_2:=\mat{\omega}{0}{0}{\bar{\omega}}, u_3=\mat{0}{1}{-N}{0}, u_4=\mat{0}{\omega/N}{-\bar{\omega}}{0}.
\end{equation}
Hence $\rs'$ is maximal. For $p\not\!| \,a$, the matrix $A$ from \eqref{E:defA} is in $\Mat_2(\Z_p)$ and so gives an isomorphism $M\mapsto AMA^{-1}$ from $(\Rz)_p\to \rs_p'$. Hence $(\Rz)_p$ is maximal for $p\not|a$. 

There exists a form $\tilde{Q}=\smat{2\tilde{a}}{\tilde{b}}{\tilde{b}}{2\tilde{c}}$ properly equivalent to $Q$ with $\gcd(2a,\tilde{a})=1$ (see \cite[p. 25,35]{Co}, for example). Applying Lemma \ref{L:isoB} to the pair $Q$ and $ \tilde{Q}$ gives $\Rz\cong \rs_{\tilde{Q}\tau}$. Hence for $p|a$ we can apply the paragraph above to $\rs_{\tilde{Q}\tau}$ to conclude $(\Rz)_p$ is maximal. 
\end{proof}

\begin{lem}$(\Rz)_2$ is  maximal. 
\end{lem}
\begin{proof}[Proof of Lemma]
Note $\gcd(2a,b)=1$ because $N$ is prime and $b$ is odd. Define $U:=\smat{1}{0}{-2cx-by}{1}$ and $ V:=\smat{y}{-b}{x}{2a}$ where $x,y\in\Z$ such that $2ay+bx=1$. Then $UQV=Q'$ where $Q'$ was defined in \eqref{E:defA}. Define $\hat{H}:={}^TU^{-1}HU^{-1}$, $\hat{\B}:=\set{ M \in \End_\Q(X_{Q'\tau})\col {}^T\!\bar{M}\hat{H} = \hat{H} M^\iota }$, and $\hat{\rs}:=\hat{\B}\cap \End(X_{Q'\tau})$. The period matrix $\Pi':=[Q'\tau,\id]$ satisfies $\Pi'= U\Pi_z \tilde{V}$ where $\tilde{V}:=\smat{V}{0}{0}{U^{-1}}\in \Mat_4(\Z)$. Hence the map $M\mapsto UMU^{-1}$ from $\Rz\to \hat{\rs}$ is an isomorphism over $\Z$. Therefore $(\Rz)_p \cong \hat{\rs}_p$ for all primes $p$. We will show $\hat{\rs}_2$ is maximal. 

By Lemma \ref{L:isoB} and the isomorphism $\B\cong\B'$, a basis for $\B$ is given by the  set $\set{A^{-1}u_iA}$ with $A$ defined in \eqref{E:defA} and $u_i$ in \eqref{E:uibasis}. Hence by above the set $\set{v_i:=UA^{-1}u_iAU^{-1}}$ gives a basis for $\hat{\B}$ over $\Q$. Replace $v_i$ with $2av_i$ for $i=2,3$ and $v_4$ by $2aNv_4$. Then explicitly,
\begin{align*}
	v_1&=\mat{1}{0}{0}{1} & v_2&=\mat{2a\omega}{0}{-Nx(b_1+2\omega)}{2a\bar{\omega}}\\
	v_3&=\mat{2aNx}{4a^2}{-N(Nx^2+1)}{-2aNx}&v_4&=\mat{2aNx\omega}{4a^2\omega}{-N(Nx^2\omega+\bar{\omega})}{-2aNx\omega}.
\end{align*}
By Lemma \ref{L:end} we see $v_i\in \hat{\rs}$, $i=1,\dots, 4$. To prove $\hat{\rs}_2$ is maximal we will use the elements $\{v_i\}$ to construct a basis of $\hat{\rs}_2$ whose discriminant is a unit modulo $(\Z_2)^2$. 

Associate any matrix $M:=(m_{ij}+n_{ij}\omega) \in \Mat_2(\Q(\omega))$ with $m_{ij}, n_{ij}\in \Q$ to the vector 
\[
\vec{v}_M:= {}^T(m_{11}, n_{11}, m_{12}, n_{12}, m_{21}, n_{21}, m_{22}, n_{22}) \in \Q^8.
\]
 Denote the vector $\vec{v}_{v_i}$ by $\vec{v_i}$ for simplicity. Let $M_{\text{bas}}\in \Mat_{8\times4}(\Z)$ be the matrix whose $i$-th column is $\vec{v}_i$ for $i=1,\dots, 4$. Given $M\in \Mat_2(K)$, $M \in \hat{\B}$ if and only if
\begin{equation}\label{E:MinB}
\vec{v}_M=M_{\text{bas}}\cdot \vec{\alpha}_M
\end{equation}
 for some $\vec{\alpha}_M\in \Q^4$. Moreover $M:=(m_{ij}+n_{ij}\omega)$ is in $\End(Q'\tau)$ if and only if 
 \begin{equation}\label{E:endcond}
 m_{11},  n_{11}, m_{12}, Nn_{12}, \frac{m_{21}-b_1n_{21}}{N}, n_{21}, m_{22}, n_{22} \in \Z
 \end{equation}
 by \eqref{E:Rprime}.
  Let $M_{\text{end}}\in \Mat_{8\times8}(\Q)$ be the matrix which describes the conditions in \eqref{E:endcond} so that $M \in \End(Q'\tau)$ if and only if $M_{\text{end}}\cdot\vec{v}_M\in \Z^8$. Therefore the elements $M$ of $\hat{\rs}$ correspond precisely under \eqref{E:MinB} to $\vec{\alpha}_M\in \Q^4$ such   that 
  \begin{equation}\label{E:Rhatcond}
  M_{\text{end}}\cdot M_{\text{bas}}\cdot \vec{\alpha}_M\in \Z^8.
  \end{equation}
   To show the discriminant of $\hat{\rs}$ is $1\bmod (\Z_2)^2$ amounts to finding solutions $\vec{\beta}\in \Z^4$ such that $M_{\text{end}}\cdot M_{\text{bas}}\cdot\vec{\beta}\equiv 0\bmod 4$. (Then $\vec{\alpha}:=\vec{\beta}/4$ satisfies \eqref{E:Rhatcond}.) Three linearly independent solutions for $\vec{\alpha}$ are given by the vectors
\[
 \vec{\alpha}_5:={}^T(0,0,1,0)/2,\quad \vec{\alpha}_6:={}^T(0,1,0,1)/2, \quad\text{ and } \vec{\alpha}_7:={}^T(2,0,1,0)/4.
 \]
  Therefore  $\vec{v}_i:= M_{\text{bas}}\cdot \vec{\alpha}_i$ gives an element  in $ \hat{\rs}$ for $i=5,6,7$. Consider the set $S:=\set{\vec{v}_1,\vec{v}_2, \vec{v}_3, \vec{v}_4}$. Observe the relations 
  \[
  \vec{v}_5=\vec{v}_3/2, \qquad \vec{v}_7=(\vec{v}_1+\vec{v}_5)/2,\qquad \text{  and } \vec{v}_6=(\vec{v}_2+\vec{v}_4)/2.
  \]
   These imply $\vec{v}_5$ generates $\vec{v}_3$, while $\vec{v}_1$ and $\vec{v}_7$ generate $\vec{v}_5$, and finally  $\vec{v}_2$ and $\vec{v}_6$ generate $\vec{v}_4$.  Accordingly, replace $\vec{v}_3$ and $\vec{v}_4$ in $S$ with $\vec{v}_6$ and $ \vec{v}_7$ so that $S=\set{\vec{v}_1, \vec{v}_2, \vec{v}_6, \vec{v}_7}$. Now $S$ is a set of linearly independent vectors over $\Z$ and contained in $\hat{\rs}$, hence a basis.  A computation (using PARI/GP \cite{Pari}) of the discriminant of $\hat{\rs}$ with respect to this basis shows it is $D^2\cdot N^2\cdot a^6$. This is a unit modulo $(\Z_2)^2$ since we may assume $a$ is odd. Hence $\hat{\rs}_2$ is maximal. 

\end{proof}

This concludes the proof that $\Rz$ is a maximal order. \end{proof}

The next step is to prove $\Rz$ is the right order of an explicit ideal in $\Dalg$. We first recall a result of Pacetti which constructs Siegel points from certain ideals of $\Dalg$.

\section{Split-CM points and right orders in $\Dalg$}\label{S:splitcm}

In this section we identify $\Rz$ with an explicit right order in $\Dalg$.  Let $\Lor$ be a maximal order of $\Dalg$ such that there exists $u\in \Lor$ with $u^2=D$. (Such an order must exist by Eichler's mass formula).  Two left $\Lor$-ideals $I$ and $I'$ are in the same \emph{class} if there exists $b\in \Dalg^\times$ such that $I = I'b$. The number $n$ of left $\Lor$-ideal classes is finite and independent of the choice of maximal order $\Lor$. Let $\mathcal{I}$ be the set of $n$ left $\Lor$-ideal classes, and recall $\ror$ is the set of conjugacy classes of maximal orders in $\Dalg$. (Equivalently, $\ror$ is the set of conjugacy classes of right orders with respect to $\Lor$, taken without repetition.)   The cardinality $t$ of $\ror$ is less than or equal to $n$ and is called the \emph{type number}.  

Recall $\Dalg \cong (D,-N)_\Q$ and let $1, u, v, uv$ be a basis for $\Dalg$ where $u^2=D$, $v^2=-N$, and $uv=-vu$.  Define the $\Z$-module
\begin{equation}\label{E:Iz}
	I_z:= \bigg< \bigg( \frac{b_1-u}{2a_1N}\bigg)av, \bigg(\frac{b_1-u}{2a_1N}\bigg)\bigg(\frac{N+bv}{2}\bigg), \frac{b-v}{2}, -a\bigg>_\Z.
\end{equation}
It is proven in \cite[p. 369-372]{Pa} that $I_z$ is a left ideal for a maximal order $\Lor_{\A, [\Nn]}$ which is independent of the class representative of $[\Nn]$ and of the form $Q$, and contains the element $u$. Let $\rz$ denote the right order of $I_z$. It is maximal because $\Lor_{\A, [\Nn]}$ is maximal.  

We will show that the right order $R_z$ has a natural identification with the maximal order $\Rz$. To do this, we recall a result of \cite{Pa} which associates ideals of $\Dalg$ to Siegel points. Namely, let $(I_R, R)$ be a pair consisting of a left $\Lor$-ideal $I_R$ with maximal right order $R$.   
 Define the $4$-dimensional real vector space $V:= \Dalg\otimes_\Q \R$, so that    $V/I_R$ is a real torus. The linear map
\begin{align*}
	J: V&\to V\\
	x &\mapsto \frac{u}{\sqrt{|D|}} \cdot x
\end{align*}
induces a complex structure on $V$. Hence the data $(V/I_R, J)$ determines a $2$-dimensional complex torus.
 Define a map $\er: V\times V\to \R$ by 
\[
	\er(x,y):= \Tr(u^{-1}x\bar{y})/\norm(I_R),
\]
where $\norm(I_R)$ is the norm of the ideal $I_R$ and the `bar' denotes conjugation in $\Dalg$. It is straightforward to check that $\er$ is alternating, satisfies $\er(Jx,Jy)=\er(x,y)$ for all $x,y \in V$, is integral on $I_R$, and that the form $\hr:V\times V \to \C$ defined by
\begin{equation}\label{E:defH}
	\hr(x,y):= \er(Jx,y)+ i\er(x,y), \qquad x,y \in V
\end{equation}
is positive definite (see \cite{Pa} for details). Thus $\er$ is a Riemann form and so there exists a symplectic basis $\{x_1, x_2, y_1, y_2\}$ of $I_R$ with respect to $\er$. The matrix $E_R$ of $\er$ with respect to this basis  has determinant
\[
	\det(E_R) = \norm(I_R)^{-4} \norm(u)^{-2} \disc(I_R),
\]
where we have used the fact that  $\disc(I_R) = (\det(u_i u_j))_{ij}$ for any basis $\{u_1, \dots, u_4\}$ of $I_R$. But the fact that $R$ is maximal implies  $\disc(I_R)= D^2 \,\norm(I_R)^4$ \cite{Pi}, \cite[Proposition 32]{Pa}, hence $\det(E_R)=1$.  This implies    $\er$ is of type $1$, its matrix is $E_R= \smat{0}{\id}{-\id}{0}$, and  $\hr$ is a principal positive definite Hermitian form. 

The conclusion is that  the data 
$(I_R, J, E_R)$
determines a Siegel point in $\h_2/Sp_4(\Z)$. The action of a $\g \in Sp_4(\Z)$ on $(I_R, J, E_R)$ is given as a $\Z$-linear isomorphism $I_R\to\g(I_R)$, which sends $J\to \g^{-1}\circ J\circ\g$, and $\er\to \er\circ\g$. 

Left $\Lor$-ideals with the same right order class determine equivalent Siegel  points under this construction \cite[p. 364]{Pa}. In other words, there is a well-defined map
\[
	\ror \To \h_2/\Sp.
\]
This can be seen as follows.  Let $I$ and $I'$ be two left $\Lor$-ideals with the same right order class $[R]$. Assume first that they are equivalent, that is, $I = I'b$ for some $b\in \Dalg^\times$. Then multiplication on the right by $b$ determines a $\Z$-linear isomorphism
\begin{align*}
	\g: I&\To I'\\
		x&\mapsto x\cdot b.
\end{align*}
Furthermore
\[
	E(\g(x),\g(y)) = \frac{\Tr(u^{-1}x\cdot b (\overline{y\cdot b}))}{\norm(I)} = E(x,y) \cdot \frac{\norm(b)}{\norm(I)} = E'(x,y),
\]
and since $J$ is a multiplication on the left, and $b$ on the right, clearly $\g^{-1}\circ J\circ \g=J$. Therefore $(I,J,E)\sim (I',J,E')$ for $I\sim I'$.  Now suppose $I$ and $I'$ are not equivalent. Then $uI$ has the same left order and right order class as $I$ but is not equivalent to $I$ (see Lemmas \ref{L:qalg1} and \ref{L:qalg3} below). Since there are at most two classes of left $\Lor$-ideals with the same right order class, it must be that $uI\sim I'\sim uIu^{-1}$. It is straightforward to check that the map from $I$ to $uIu^{-1}$ via conjugation by $u$ gives $ (I, J,E) \sim (uIu^{-1}, J, E)$ and so by the above case, $(I, J, E)\sim (I', J, E')$.

The ideal $I_z$ in \eqref{E:Iz} corresponds to the Siegel point $z$ under this construction. This is left as an exercise in \cite{Pa} but can be seen as follows. Let $\{x_1, x_2, y_1, y_2\}$ denote the basis, taken in order, of $I_z$ given in \eqref{E:Iz}.  A straightforward calculation done by Pacetti shows $\{x_1, x_2, y_1, y_2\}$ is symplectic with respect to $\mathcal{E}$, and of principal type. Then  $\{y_1, y_2\}$ is a basis for the complex vector space $(V, J)$, and the period matrix for the complex torus $(V/I_z, J)$ is the coefficient matrix of the basis of $\{x_1, x_2, y_1, y_2\}$ in terms of $\{y_1, y_2\}$. It suffices to show this period matrix  is $\Pi_z:=[z, \id]$.  Thus one needs to verify
\begin{align*}
	x_1 &= 2a\tilde{\tau} y_1 + b \tilde{\tau} y_2\\
	x_2 &= b\tilde{\tau} y_1 + 2c\tilde{\tau} y_2,
\end{align*}
where $\tilde{\tau}:=\frac{-b_1+\sqrt{|D|}J}{2a_1N}$ is given by the complex multiplication $J$. This is a simple calculation using the relations $D=b_1^2-4a_1c_1N$ and $-N=b^2-4ac$.  

Note this construction determines an isomorphism $\sigma: I_z\To L_z$ by 
\[
	x_1\mapsto \symb{2a}{b}\tau,\quad x_2\mapsto \symb{b}{2c}\tau,\quad y_1\mapsto \symb{1}{0},\quad y_2\mapsto \symb{0}{1},
\]
which maps $J\mapsto i$. In particular $\hh_{R_z}(x,y) = \hh_z\big|_{L_z\times L_z} (\sigma(x), \sigma(y))$ for all $x,y \in I_z$. 

The elements of $R_z$ and $\Rz$ can now be related as follows. 
Any $b\in R_z$ preserves $I_z$ (on the right) as well as the complex structure $J$ and hence  defines an endomorphism $f_b$ of $X_z$. Likewise, any $M \in \Rz$ defines an endomorphism $f_M$ of the torus $X_z$ by definition. We claim these rings give the same endomorphisms of $X_z$:

\begin{prop}\label{P:RisR}
As endomorphisms, $R_z$ is identified with $\Rz$. 
\end{prop}
\begin{proof}
Suppose $f_b\in \End(X_z)$ for some $b\in \rz$. To show $f_b$ comes from $\Rz$, it suffices to show $\rho_r(f_b)$ preserves $H_z\big|_{L_z\times L_z}$. Equivalently by the map $\sigma$ it suffices to show
\[
	\hh_{R_z}( x\cdot b, y)= \hh_{R_z}(x,y\cdot b^\iota).
\]
But this is immediate since $\Tr(u^{-1}(xb)\bar{y}) = \Tr(u^{-1}x(\bar{\bar{b}}\bar{y})= \Tr(u^{-1}x(\overline{y \bar{b}}))$ and $\bar{b}=b^\iota$ in $\Dalg$. Therefore as endomorphisms $R_z$ is contained in $\Rz$. Conversely any $f_M \in \End(X_z)$ for $M\in \Rz$ defines a linear map from $I_z$ to itself which commutes with the complex structure $J$, hence corresponds to an element in $R_z$.  \end{proof}

\begin{cor}\label{C:RzisRz}
$\Rz$ is isomorphic to the maximal right order $R_z$ in the quaternion algebra $\B$, and this map sends $\frac{1+QS}{2}\mapsto \frac{1+v}{2}$.
\end{cor}
\begin{proof}
The first part follows immediately from the proposition. Regarding the embedding, the rational representation in $\Mat_4(\Z)$ of the endomorphism $\frac{1+QS}{2}\in \rz$ is
\[
\left(\begin{array}{cccc}\frac{b+1}{2} & c & 0 & 0 \\-a & \frac{1-b}{2} & 0 & 0 \\0 & 0 & \frac{1-b}{2} & a \\0 & 0 & -c & \frac{b+1}{2}\end{array}\right).
\]
Its action on the basis $x_1, x_2, y_1, y_2$ of $I_z$ shows immediately that it is the linear transformation given by multiplication on the right by $\frac{1+v}{2}$. 
\end{proof}

\section{Formula for the central value $L(\psin, 1)$}\label{S:centralvalue}

In this section we prove Theorems \ref{T:firsthm}, \ref{T:secondthm} and \ref{T:mainthm}.

\begin{proof}[Proof of Theorems \ref{T:firsthm} and \ref{T:secondthm}]
Fix $[\A] \in Cl(\Oo_K)$, $\Nn \subset \Oo_K$ a prime ideal of norm $N$, $\tau:= \Tan$. Throughout the rest of this section, fix $z:=Q\tau$ and $z':=Q'\tau$ where $Q, Q'$ are binary quadratic forms of discriminant $-N$.  Define  
\begin{align}
	\Ups_1: \set{ Q\tau \col [Q] \in Cl(-N)}/\Sp &\To \ror_N\\\nonumber
		[Q]\tau &\mapsto [R_{Q\tau}]
\end{align}
Given an  $R_{Q\tau}$, let $\phi_Q : \Oo_L\hookrightarrow R_{Q\tau}$ be the optimal embedding defined in Lemma \ref{L:Rzorder} and Corollary \ref{C:RzisRz}.  Define a second map 
\begin{align}
	\Ups_2: Cl(-N) &\To \Phi_\ror/-\\\nonumber
		[Q] &\mapsto [ \phi_Q: \Oo_L \hookrightarrow R_{Q\tau}].
\end{align}
We will start by showing that the maps $\Ups_1$ and $\Ups_2$ are well-defined. First note  $\Ups_1$ is injective: if $R_{Q\tau} \sim R_{Q'\tau}$ in $B$, then  we saw in the last section that Pacetti's map $\ror \To \h_2/\Sp$ sends $R_{Q\tau} \mapsto Q\tau$.  After proving the maps are well-defined, we will prove $\Ups_2$ is a bijection and independent of the choice of representative $\A$ of $[\A]$. This will simultaneously prove Theorems  \ref{T:secondthm} and \ref{T:firsthm}.

\begin{lem}
If $z\sim z'$ in $\h_2/\Gamma_\theta$, then $R_z \sim R_{z'}$ in $\ror$.
\end{lem}
\begin{rem}\label{R:equivQ}
Note that if $Q\sim Q'$ with $Q = {}^TA Q' A$ for some $A\in SL_2(\Z)$, then $Q\tau \sim Q'\tau$ as Siegel points via the matrix $\smat{{}^TA}{0}{0}{A^{-1}} \in \Gamma_\theta$. 
\end{rem}

\begin{proof}

Recall  $z\sim z'$ in $\h_2/Sp_4(\Z)$ if and only if the  abelian varieties $(X_z, H_z)$  and $(X_{z'}, H_{z'})$ are isomorphic. Write $X, X', H, H'$ for $X_z, X_{z'}, H_z, H_{z'}$, respectively. Suppose $f:X \To X' $ is an isomorphism of $(X, H)$ with $(X', H')$, so that $H'(f(x),f(y))= H(x,y)$ for all $x, y \in \C^2$.  We claim the isomorphism 
\begin{align}
	\End(X) &\To \End(X')\\\nonumber
	\alpha &\mapsto f \circ \alpha \circ f^{-1}
\end{align}
induces an isomorphism of $\Rz$ and $\mathscr{R}_{z'}$. This follows immediately from the calculation
\begin{align*}
	H'( f\circ \alpha \circ f^{-1}(x),y) &= H( \alpha(f^{-1}(x)), f^{-1}(y))\\
	&= H( f^{-1}(x), \alpha^\iota (f^{-1}(y))) &&\text{(since $\alpha \in \Rz$)}\\
	&= H'(x, f(\alpha^\iota(f^{-1}(y)))) \\
	&= H'(x, (f\circ\alpha\circ f^{-1})^\iota).
\end{align*}
The last equality follows because, as a matrix, $\rho_a(f)^\iota=\rho_a(f)^{-1}\det(\rho_a(f))$ and so the determinants in $(f\circ\alpha\circ f^{-1})^\iota$ cancel out. 
Therefore $\mathscr{R}_{z'} = f\circ \Rz\circ f^{-1}$ and so by Proposition \ref{P:RisR}, $\rz \sim R_{z'}$ in $\B$. 
\end{proof}

\begin{lem}
If $Q\sim Q'$ in $Cl(-N)$, then the corresponding optimal embeddings $\frac{v+1}{2} \hookrightarrow R_z$ and $\frac{v+1}{2}\hookrightarrow R_{z'}$ are equivalent.  
\end{lem}
\begin{proof}
Suppose $Q\sim Q'$ with $Q' = A Q {}^T\! A$ for some $A \in SL_2(\Z)$. Then by Lemma  \ref{L:isoB}, the map $\Rz \to \mathscr{R}_{z'}$  by $M \mapsto AMA^{-1}$ is a $\Z$-algebra isomorphism, and extends to a $\Q$-algebra isomorphism from $\B \to \B'$. In particular it sends $QS \mapsto A(QS)A^{-1} = Q'S$. By Corollary \ref{C:RzisRz}, this induces a $\Z$-algebra isomorphism of $R_z \to R_{z'}$ which sends $v$ to $v$, and extends to a $\Q$-algebra automorphism of $\Dalg$. Hence by the Skolem-Noether theorem, the map $R_z \to R_{z'}$ must be conjugation by some unit of $\Dalg$. 
\end{proof}

We now turn to proving $\Ups_2$ is a bijection. The following six lemmas  will be needed to prove $\Ups_2$ is injective. Let $\Qq$ denote the ideal in $L$ which corresponds to $Q$. 
\begin{lem}\label{L:Izomod}
\[
	I_z \cong \bQq \oplus \bQq
\]
as right $\Oo_L$-modules. 
\end{lem}
\begin{proof}[Proof of Lemma]


Define $v_1:= x_1, v_2:=x_2, v_3:=y_1, v_4:=-y_2$ where $x_i, y_j$ is the basis of $I_z$ defined in Section \ref{S:heegel}. The $\{v_i\}$ also form a basis for $\Iz$. The  map $f: \Iz \To \bar{\Qq} \oplus \bar{\Qq}$ defined by
\begin{align*}
	v_1 &\mapsto (a,0) &v_2 &\mapsto (\frac{b-\sqrt{-N}}{2},0)\\
	v_4 &\mapsto (0,a) &v_3 &\mapsto (0, \frac{b-\sqrt{-N}}{2})
\end{align*}
and extended $\Z$-linearly is an isomorphism of $\Z$-modules. To show it is an $\Oo_L$-module isomorphism, it suffices to show
\[
	f\bigg( v_i \bigg(\frac{b+v}{2}\bigg)\bigg) = f(v_i) \bigg(\frac{b+\sqrt{-N}}{2}\bigg) \quad \text{for all } \: i=1,2,3,4.
\]
For this, use the identities: 
\begin{align*}
 v_1 \bigg(\frac{b+v}{2}\bigg) &= bv_1 - av_2 & &v_3 \bigg(\frac{b+v}{2}\bigg) = cv_4\\
 v_2 \bigg(\frac{b+v}{2}\bigg) &= cv_1& &v_4 \bigg(\frac{b+v}{2}\bigg) = -av_3 + bv_4.
 \end{align*}
\end{proof}

\begin{lem}\label{L:Izxomod}
Suppose $S:= I_z x$ where $x \in \Dalg^\times$ commutes with $\emb$. Then 
\[
	S\cong \bQq \oplus \bQq,
\]
as right $\Oo_L$-modules. 
\end{lem}
\begin{proof}[Proof of Lemma]
By Lemma \ref{L:Izomod} and the hypotheses on $x$, the composition from $S \to \bQq\oplus \bQq$ given by $g(v_ix):= f(v_i)$
is an isomorphism of $\Oo_L$-modules. 
\end{proof}

\begin{lem}\label{L:QequivQ'}
Suppose $\bQq \oplus \bQq \cong \bQq' \oplus \bQq'$ as right $\Oo_L$-modules, and $h(-N)$ is odd. Then 
\[
	Q \sim Q'
\] 
in $Cl(-N)$. 
\end{lem}
\begin{proof}[Proof of Lemma]

By  a classical theorem of Steinitz \cite[Theorem 1.6]{Mi}, $\bQq \oplus \bQq \cong \bQq' \oplus \bQq'$
as right $\Oo_L$-modules if and only if $[\bQq']^2 = [\bQq]^2$
as classes in the ideal class group of $\Oo_L$. This is if and only if $[\bQq'/\bQq]^2= [\text{id}]$ where id is the identity class. But since the class number $h(-N)$ is odd, this implies $[\Qq] = [\Qq']$
 in $Cl(-N)$.
\end{proof}

The next three lemmas we need are general results for quaternion algebras. Assume for Lemmas \ref{L:qalg1}, \ref{L:qalg2}, and \ref{L:qalg3} below that $B$ is a quaternion algebra ramified precisely at $\infty$ and a prime $p$. In addition, assume  $M$ and $R$ are maximal orders and there exists $u\in M$ such that $u^2=-p$. 

\begin{lem}\label{L:qalg1}
\[
	uMu^{-1} = M.
\]
\end{lem}
\begin{proof}
This is clear locally at primes $q\neq p$ because $u^{-1} = -u/p$. This is also clear locally at $p$ because  there is a unique maximal order in the division algebra $B_p$ (see \cite[Theorem 6.4.1, p.208]{Re} or \cite{Vig} for example).
\end{proof}

\begin{lem}\label{L:qalg2}
Suppose $I, I'$ are left $M$-ideals with right order $R$. In addition assume $R$ admits an embedding of a ring of integers $\Oo$ of some imaginary quadratic field. Set $J:=I(I')^{-1}$. Then
\[
	JI' \cong I'
\]
as right $\Oo$-modules. 
\end{lem}
\begin{proof}
First note $J$ is a bilateral $M$-ideal.   Since $u\in M$, $uM=Mu$ by Lemma \ref{L:qalg1} and so is a principal $M$-ideal of norm $p$. Hence it is the unique integral bilateral $M$-ideal  of norm $p$, and so every bilateral $M$-ideal is equal to $uM\cdot m$ for some $m\in \Q$  \cite[Proposition 1, p. 92]{Ei2}. In particular, this implies the bilateral $M$-ideals are principal. Therefore $J= tM = Mt$ for some $t\in B^\times$, and the map,
\begin{align*}
	f: I' &\To JI'\\
	w &\to tw
\end{align*}
is a $\Z$-module isomorphism. Since the multiplication by $t$ is on the left, $f$ is an isomorphism of right $\Oo$ modules. 
\end{proof}

\begin{lem}\label{L:qalg3}
Suppose $I$ is a left $M$-ideal with right order $R$. Then $uI$ is also a left $M$-ideal with right order $R$. Furthermore, any left $M$-ideal with right order $R$ is equivalent to $I$ or $uI$ (or both). 
\end{lem}
\begin{proof}
The right order of $uI$ is clearly $R$. The left order is $uMu^{-1} = M$ by Lemma \ref{L:qalg1}. 

Suppose $J$ is any left $M$-ideal with right order $R$. The ideal $I^{-1}J$ is $R$-bilateral, hence
\[
	I^{-1}J = \mathcal{P}^i m, \qquad i=0,1, \: m\in \Q
\]
where $\mathcal{P}$ is the unique bilateral $R$-ideal of norm $p$ \cite[Proposition 1, p. 92]{Ei2}.

If $I^{-1}J$ is principal, then $I\sim J$.  Otherwise $i=1$. Then since the ideal $I^{-1}uI$ is $R$-bilateral of norm $p$, by uniqueness $I^{-1}uI=\mathcal{P}$ and so
\[
	I^{-1}J = I^{-1}uI\cdot m.
\]
Multiplying through by $I$ we see $J \sim uI$ as left $M$-ideals. 
\end{proof}

Now the injectivity of $\Ups_2$ can be proven. 

\begin{prop}
Suppose $(R_z, \frac{v\pm 1}{2}) \sim (R_{z'}, \frac{v\pm 1}{2})$. Then $Q\sim Q'$ in $Cl(-N)$. 
\end{prop}
\begin{proof}

The assumption $(R_z, \frac{v\pm 1}{2}) \sim (R_{z'}, \frac{v\pm 1}{2})$ implies there exists $x\in \Dalg^\times$ such that 
\[
	x^{-1}\rz x = R_{z'}
\]
and $r \in R_{z'}^\times$ such that
\[
	(xr)^{-1} \bigg(\emb \bigg) xr = \emb.
\]

The proof is broken up into two cases.

\subsection*{Case 1} Assume $I_z\sim I_{z'}$. Then $I_zx \sim I_{z'}$ and they both have right order $R_{z'}$. Set $J:=I_zxI_{z'}^{-1}$. Then $JI_{z'} \cong I_{z'}$ as right $\Oo_L$-modules by Lemma \ref{L:qalg2}. Combining with Lemma \ref{L:Izomod}
applied to $I_{z'}$ implies
\[
	JI_{z'} \cong \bQq' \oplus \bQq'
\]
as right $\Oo_L$-modules. 

On the other hand, $JI_{z'} = I_zx$. Since $r$ is a unit, $I_zx=I_zxr$, so replacing $x$ by $xr$ if necessary we may assume $r=1$ and $x^{-1}\big(\emb\big)x = \emb$. Lemma \ref{L:Izxomod} applied to $I_zx$ gives
\[
	JI_{z'} \cong \bQq\oplus\bQq
\]
as right $\Oo_L$-modules. Hence $Q\sim Q'$ by Lemma \ref{L:QequivQ'}. 

\subsection*{Case 2} Assume $I_z\not\sim I_{z'}$. For each maximal order $R$, there can be at most two left $M$-ideal classes with right orders in the class $[R]$. Therefore since $I_{z'}$ has right order $R_{z'}\in [R_z]$, but $I_z\not\sim I_{z'}$, by Lemma \ref{L:qalg3} it must be that
\[
	uI_z \sim I_{z'};
\]
note $uI_z$ is a left $\Lor$-ideal by Lemma \ref{L:qalg1}. Then $uI_z x \sim I_{z'}$ and they have the same right order. Let $J:=uI_zxI_{z'}^{-1}$ and use the same argument from Case $1$, noting that Lemmas \ref{L:Izomod} and \ref{L:Izxomod} hold with $I_z$ replaced by $uI_z$ since the multiplication by $u$ is on the left. 
This concludes the  proof that $\Ups_2$ is injective. 
 \end{proof}

It remains to show that $\Ups_2$ is a surjection. This follows from the fact:
\begin{lem}
\[
	h(-N) = \#\Phi_\ror/- .
\]
\end{lem}
\begin{proof}[Proof of Lemma]
For $[R]\in \ror$, let $h_R(-N)$ denote the number of optimal embeddings of $\Oo_L$ into $R$, modulo conjugation by $R^\times$. 
Then
\begin{align*}
	\#\Phi_\ror/- &= \frac12 \sum_{[R]\in \ror} h_R(-N) &&\text{by definition,}\\
	&= h(-N) &&\text{by Eichler's mass formula \cite[$(1.12)$]{Gr}}.
\end{align*} 
\end{proof}
The last task is to prove the maps $\Ups_1$ and $\Ups_2$ are independent of the choice of representative $\A$ of $[\A]$. In fact we will prove a slightly stronger result regarding the right orders:
\begin{lem}\label{L:RindA}
	If $\A \sim \A'$ in $Cl(\Oo_K)$ then $R_{Q\tau_{\A\bar{\Nn}}} = R_{Q\tau_{\A'\bar{\Nn}}}$. 
\end{lem}
\begin{proof}
	The hypothesis $\A\sim\A'$ implies $\A\bar{\Nn}\sim\A'\bar{\Nn}$. Suppose $\bar{\Nn}$ corresponds to a form $[N,b, c]$. Then we can choose bases so that the products $\A\bar{\Nn}$, $\A'\bar{\Nn}$ both correspond to forms with middle coefficient congruent to $b\bmod 2N$ (see \cite[Lemma 2.3]{V2}, for example). The CM-points $\tau_{\A\bar{\Nn}}, \tau_{\A'\bar{\Nn}}$  are Heegner points of level $N$ and discriminant $D$ by construction, and by the comment above they have the same `root' $b\bmod 2N$ of $\sqrt{D\bmod 4N}$. Hence there exists $M:= \smat{\alpha}{\beta}{\gamma}{\delta} \in \Gamma_0(N)$ such that 
	\[
		M(\tau_{\A\bar{\Nn}}) = \tau_{\A'\bar{\Nn}}. 
	\]
Set
\[
	\tilde{M}:= \mat{\tilde{\alpha}}{\tilde{\beta}}{\tilde{\gamma}}{\tilde{\delta}}
\]
where
\[ 
\tilde{\alpha}:= \alpha\cdot \id, \quad \tilde{\beta}:= \beta\cdot Q, \quad\tilde{\gamma}:= \gamma\cdot Q^{-1}, \quad\tilde{\delta}:= \delta\cdot\id.
\]
It is shown in \cite[p.233]{AnMa}, for example, that $\tilde{M}\in \Gamma_\theta\subseteq \Sp$. Therefore the relation
\[
	\tilde{M}(Q\tau_{\A\bar{\Nn}}) = Q\tau_{\A'\bar{\Nn}}
\]
implies $Q\tau_{\A\bar{\Nn}} \sim Q\tau_{\A'\bar{\Nn}}$ in $\h_2/\Gamma_\theta$. Let $\tau:= \tau_{\A\bar{\Nn}}$ and $\tau':=\tau_{\A'\bar{\Nn}}$. An isomorphism $f_M: X_{Q\tau'} \To X_{Q\tau}$ is given by
\[
	{}^T(\tilde{\gamma} Q\tau + \tilde{\delta})[Q'\tau, \id] = [Q\tau, \id] \:{}^T\mat{\tilde{\alpha}}{\tilde{\beta}}{\tilde{\gamma}}{\tilde{\delta}}. 
\]
The analytic representation of this isomorphism, which we will also denote by $f_M$, is 
\[
	f_M={}^T(\tilde{\gamma} Q\tau + \tilde{\delta}) =  (\gamma \tau + \delta)\cdot \id,
\]
where recall $\gamma, \delta \in \Z$. Therefore the map 
\begin{align*}
	\End(X_{Q\tau'}) &\To \End(X_{Q\tau})\\
		A &\mapsto f_M A f_M^{-1} =A
\end{align*}
is  the identity map, hence $\End(X_{Q\tau'}) = \End(X_{Q\tau})$. Moreover the equivalence
\[
	{}^T\bar{A} H_{Q\tau} = H_{Q\tau} A^\iota \qquad \Leftrightarrow\qquad  {}^T\bar{A} Q^\iota = Q^\iota A^\iota
\]
implies the relation on the left hand side is independent of $\tau$. Hence $R_{Q\tau} = R_{Q\tau'}$. 
\end{proof}

It follows immediately since $R_{Q\tau} = R_{Q\tau'}$ that the maps $\Ups_1$ and  $\Ups_2$ are independent of the choice of representative $\A$ of $[\A]$. 

This completes the proofs of Theorems \ref{T:firsthm} and \ref{T:secondthm}.
\end{proof} 

 Recall the definitions of: the normalized theta values $\hthetar$ in \eqref{E:thetanorm}, the sign function $\epsar$ on the embeddings in \eqref{E:opsign}, and the twisted number of optimal embeddings $\har$ in \eqref{E:har}. The $\eta$ function in \eqref{E:thetanorm} is defined on an ideal  $\A=[a, \frac{-b+\sqrt{D}}{2}]$ of $\Oo_K$ by
\begin{equation}\label{E:etadef}
	\eta(\A):= e_{48}(a(b+3)) \cdot \eta\big(\frac{-b+\sqrt{D}}{2a}\big)
\end{equation}
where $e_n(x):=\exp(2\pi ix /n) $ for $n\in\Z$, $x \in\C$, and $\eta(z):= e_{24}(z) \prod_{n=1}^\infty (1- e^{2\pi iz})$ for $\Imm(z)>0$ is Dedekind's eta function. Using Shimura's reciprocity law it can be   shown  that the value $\hthetar$ 
is an algebraic integer (see \cite[Proposition 23, p. 355]{Pa} and \cite{HV}).

  We now prove Lemma \ref{L:pmdiff}.
\begin{proof}[Proof of Lemma \ref{L:pmdiff}]
 Theorem 31 of \cite{Pa} says that if $Q\tau_{\A\bar{\Nn}} \sim Q'\tau_{\A\bar{\Nn}}$ in $\h_2/\Gamma_\theta$, then 
\[
	\Theta_{[\A,  Q], \Nn} = \pm\Theta_{[\A,  Q'], \Nn}.
\]
The lemma therefore follows immediately by this fact and Theorem \ref{T:firsthm}. 
\end{proof}

We now prove Theorem \ref{T:mainthm}.
\begin{proof}[Proof of Theorem \ref{T:mainthm}]
The remaining step in  deriving formula \eqref{E:main} for $L(\psin,1)$ is to determine how $\theta$ behaves on equivalent split-CM points. The following is a special case of  \cite[Theorem 31]{Pa} but we give a slightly simplified  proof. 
\begin{lem}\label{L:thetapm}
Let $Q$ and $Q'$ be binary quadratic forms of discriminant $-N$. If $Q\tau \sim Q' \tau$ in $\h_2/\Gamma_\theta$, then $\theta(Q\tau) = \pm \theta(Q'\tau)$.
\end{lem}
\begin{proof}[Proof of Lemma]
Suppose  $Q\tau \sim Q' \tau$ in $\h_2/\Gamma_\theta$. Then there exists $M:= \smat{\alpha}{\beta}{\gamma}{\delta} \in \Gamma_\theta$ such that $M(Q\tau) = Q'\tau$. Recall the functional equation  for $\theta$ is
\begin{equation}\label{E:thfcnl}
	\theta(M\circ z) = \chi(M) [\det(\gamma z+ \delta)]^{1/2} \theta(z), \qquad M \in \Gamma_\theta
\end{equation}
where $\chi(M)$ is  a certain $8$th root of unity.

Then 
\[
	\frac{\theta( Q'\tau)}{\theta(  Q\tau)} = \chi(M) [\det (\gamma Q\tau + \delta)]^{1/2}. 
\]

Applying Smith Normal Form, there exists $U, V \in \SL$ such that $UQV = \smat{1}{0}{0}{N}$, and $U', V' \in \SL$ such that $U'Q'V' = \smat{1}{0}{0}{N}$. These give isomorphisms $f_U: X_{Q\tau} \rightarrow E_\tau \times E_{N\tau}$ and $f_{U'}: X_{Q'\tau} \rightarrow E_\tau \times E_{N\tau}$ respectively. From the relation $M(Q\tau)=Q'\tau$, we also get an isomorphism $f_M: X_{Q'\tau} \rightarrow X_{Q\tau}$ given by
\[
	{}^T(\gamma Q\tau + \delta)[Q'\tau, \id] = [Q\tau, \id] \: {}^T\mat{\alpha}{\beta}{\gamma}{\delta}. 
\]
Thus the composition
\[
f_U \circ f_M \circ f_{U'}^{-1}: E_\tau \times E_{N\tau} \To E_\tau \times E_{N\tau}
\]
is an automorphism, and the determinant of its analytic representation is a unit and an algebraic integer. This last fact follows from linear algebra or can be deduced directly using Lemma \ref{L:end}. 
Since $U$ and $U'$ are both in $\SL$, we get $\det(\gamma Q\tau + \delta)\in \Oo_K^\times$. Since $D<-4$ this implies $\det(\gamma Q\tau + \delta) = \pm 1$. Therefore $[\det(\gamma Q\tau + \delta)]^{1/2} = \pm \sqrt{\pm 1}$.

This proves $\frac{\theta( Q'\tau)}{\theta(  Q\tau)} = \pm\sqrt{\pm 1} \cdot \chi(M)$. But by Theorem $17$ of \cite{Pa}, the ratio of theta values on the left is an algebraic integer in the Hilbert class field of $K$.  Hence $\pm\sqrt{\pm 1} \cdot \chi(M)$ is an $8$th root of unity and an algebraic integer in the Hilbert class field of $K$,  which does not contain $i$. Therefore 
\[
	\pm\sqrt{\pm 1} \cdot \chi(M) = \pm 1. 
\]
\end{proof}

The theorem follows immediately from Lemma \ref{L:thetapm} and Theorems \ref{T:firsthm} and \ref{T:secondthm}.  
 \end{proof}

\section{Examples}\label{S:examples}

This section provides tables for two class number one examples. All calculations were done in gp/PARI \cite{Pari}. Given $D$ of class number one, for each admissable $N$ we compute a form $[N, b_1, c_1]$ corresponding to $\Nn$. We set $\A\Nn=\Nn$ since $Cl(\Oo_K)$ is trivial, and $\tau_{\A\Nn}:= \tau_\Nn := \frac{-b_1+\sqrt{D}}{2N}$ to be a Heegner point of level $N$ and discriminant $D$. We choose $[1, \frac{-b_1+\sqrt{D}}{2}]$ for a basis of $\Oo_K$ so that following definition \eqref{E:etadef},
\[
	\eta(\Nn)\eta(\Oo_K):= e_{48}^2(N(b_1+3)^2) \cdot \eta\big(\frac{-b_1+\sqrt{D}}{2N}\big) \cdot \eta\big(\frac{-b+\sqrt{D}}{2}\big).
\]
From left to right, the columns of the table are  $N$, the absolute  values of the integers ${\Theta}_{[R]}$ for each $[R] \in \ror$, the number, denoted $\#{\Theta}_{[R]}$, of classes $[Q]\in Cl(-N)$ with value $\pm{\Theta}_{[R]}$ (this equals $h_R(-N)$ by Theorem \ref{T:secondthm}), and the values $\har$.

For $D=-7$, the type number is $1$ and so  $\#{\Theta}_{[R]}= \frac12 h_R(-N)=h(-N)$  gives the $N$-th coefficient of the weight $3/2$ level $4D$ form $\frac12+\omega_R\sum_{N>0}H_D(N)q^N$ defined by the modified Hurwitz invariants $H_D(N)$ (see \cite[p. 120]{Gr} for their definition).

\begin{table}[!h]
\begin{center}
\renewcommand{\arraystretch}{1.6}
\begin{tabular}{| c | c | c | c || c | c | c | c |}\hline
 $N$  & ${\Theta}_{[R]}$ & $\#  {\Theta}_{[R]}$ & $\har$&$N$  & ${\Theta}_{[R]}$ & $\#  {\Theta}_{[R]}$ & $\har$ \\\hline
11&1&1&-1&	107&1&3&-3	\\\hline
23&1&3&-1&	127&1&5&1	\\\hline
43&1&1&1	&	151&1&7&-1		\\\hline
67&1&1&-1&	163&1&1&1	\\\hline
71&1&7&-3&	179&1&5&-3	\\\hline
79&1&5&-1 & 	191&1&13&-5		\\\hline
\end{tabular}
\vspace{.1in}
\caption{$D=-7$, $N\leq 200$, $t=1$.}\label{Ta:T1}
\end{center}
\end{table}

\vspace{2in}

\begin{table}[!h]
\begin{center}
\renewcommand{\arraystretch}{1.6}
\begin{tabular}{| c | c | c | c || c | c | c | c |}\hline
 $N$ & ${\Theta}_{[R]}$ & $\#  {\Theta}_{[R]}$ & $\har$&$N$ & ${\Theta}_{[R]}$ & $\#  {\Theta}_{[R]}$ & $\har$ \\\hline
23&0&2&2	&	103&0&3&3	\\
 &2&1&1	&	 &2&2&2\\\hline
31&0&2&2	&	163&0&1&1\\
 &2&1&-1	&	&2&0&0\\\hline
47&0&3&3	&	179&0&2&2\\
 &2&2&2	&	&2&3&1\\\hline
59&0&2&2	&	191&0&8&8\\
 &2&1&-1	&	 &2&5&1\\\hline
67&0&0&0	&	199&0&5&5\\
 &2&1&-1	&	&2&4&4\\\hline
71&0&4&4	&	 223&0&4&4\\
 &2&3&-3	&	&2&3&3\\\hline
\end{tabular}
\vspace{.1in}
\caption{$D=-11$, $N\leq250$, $t=2$. }\label{Ta:T1}

\end{center}
\end{table}

\section*{Acknowledgments}

I am deeply grateful to Fernando Rodriguez Villegas for his continuing guidance and support and for sharing his ideas that have enriched this work. I would also like to thank John Voight and Ariel Pacetti for helpful discussions on this subject. Thanks to Jeffrey Stopple for his careful reading of the manuscript.
 This research was partially funded by the Donald D. Harrington Endowment Fellowship and a Wendell Gordon Endowed Fellowship at the University of Texas at Austin.

\bibliographystyle{alpha}
\bibliography{mybibliography}

\begin{thebibliography}{{PAR}08}

\bibitem[AM75]{AnMa}
A~N Andrianov and G~N Maloletkin.
\newblock Behavior of theta series of degree $ n$ under modular substitutions.
\newblock {\em Mathematics of the USSR-Izvestiya}, 9(2):227--241, 1975.

\bibitem[BL04]{BL}
Christina Birkenhake and Herbert Lange.
\newblock {\em Complex abelian varieties}, volume 302 of {\em Grundlehren der
  Mathematischen Wissenschaften [Fundamental Principles of Mathematical
  Sciences]}.
\newblock Springer-Verlag, Berlin, second edition, 2004.

\bibitem[Cox89]{Co}
D.~A. Cox.
\newblock {\em Primes of the form {$x\sp 2 + ny\sp 2$}}.
\newblock A Wiley-Interscience Publication. John Wiley \& Sons Inc., New York,
  1989.
\newblock Fermat, class field theory and complex multiplication.

\bibitem[Eic66]{Ei1}
Martin Eichler.
\newblock {\em Introduction to the theory of algebraic numbers and functions}.
\newblock Translated from the German by George Striker. Pure and Applied
  Mathematics, Vol. 23. Academic Press, New York, 1966.

\bibitem[Eic73]{Ei2}
M.~Eichler.
\newblock The basis problem for modular forms and the traces of the {H}ecke
  operators.
\newblock In {\em Modular functions of one variable, {I} ({P}roc. {I}nternat.
  {S}ummer {S}chool, {U}niv. {A}ntwerp, {A}ntwerp, 1972)}, pages 75--151.
  Lecture Notes in Math., Vol. 320. Springer, Berlin, 1973.

\bibitem[Gro80]{Gro}
Benedict~H. Gross.
\newblock {\em Arithmetic on elliptic curves with complex multiplication},
  volume 776 of {\em Lecture Notes in Mathematics}.
\newblock Springer, Berlin, 1980.
\newblock With an appendix by B. Mazur.

\bibitem[Gro84]{Gr}
B.~H. Gross.
\newblock Heegner points on {$X\sb 0(N)$}.
\newblock In {\em Modular forms ({D}urham, 1983)}, Ellis Horwood Ser. Math.
  Appl.: Statist. Oper. Res., pages 87--105. Horwood, Chichester, 1984.

\bibitem[Gro87]{Gro2}
Benedict~H. Gross.
\newblock Heights and the special values of {$L$}-series.
\newblock In {\em Number theory ({M}ontreal, {Q}ue., 1985)}, volume~7 of {\em
  CMS Conf. Proc.}, pages 115--187. Amer. Math. Soc., Providence, RI, 1987.

\bibitem[Hec59]{He}
Erich Hecke.
\newblock {\em Mathematische {W}erke}.
\newblock Herausgegeben im Auftrage der Akademie der Wissenschaften zu
  G\"ottingen. Vandenhoeck \& Ruprecht, G\"ottingen, 1959.

\bibitem[HI80]{HI1}
Ki-ichiro Hashimoto and Tomoyoshi Ibukiyama.
\newblock On class numbers of positive definite binary quaternion {H}ermitian
  forms.
\newblock {\em J. Fac. Sci. Univ. Tokyo Sect. IA Math.}, 27(3):549--601, 1980.

\bibitem[HI81]{HI2}
Ki-ichiro Hashimoto and Tomoyoshi Ibukiyama.
\newblock On class numbers of positive definite binary quaternion {H}ermitian
  forms. {II}.
\newblock {\em J. Fac. Sci. Univ. Tokyo Sect. IA Math.}, 28(3):695--699 (1982),
  1981.

\bibitem[HI83]{HI3}
Ki-ichiro Hashimoto and Tomoyoshi Ibukiyama.
\newblock On class numbers of positive definite binary quaternion {H}ermitian
  forms. {III}.
\newblock {\em J. Fac. Sci. Univ. Tokyo Sect. IA Math.}, 30(2):393--401, 1983.

\bibitem[HK86]{HK1}
Ki-ichiro Hashimoto and Harutaka Koseki.
\newblock Class numbers of positive definite binary and ternary unimodular
  {H}ermitian forms.
\newblock {\em Proc. Japan Acad. Ser. A Math. Sci.}, 62(8):323--326, 1986.

\bibitem[HK89]{HK2}
Ki-ichiro Hashimoto and Harutaka Koseki.
\newblock Class numbers of positive definite binary and ternary unimodular
  {H}ermitian forms.
\newblock {\em Tohoku Math. J. (2)}, 41(2):171--216, 1989.

\bibitem[HV97]{HV}
Farshid Hajir and Fernando~Rodriguez Villegas.
\newblock Explicit elliptic units. {I}.
\newblock {\em Duke Math. J.}, 90(3):495--521, 1997.

\bibitem[Mil71]{Mi}
John Milnor.
\newblock {\em Introduction to algebraic {$K$}-theory}.
\newblock Princeton University Press, Princeton, N.J., 1971.
\newblock Annals of Mathematics Studies, No. 72.

\bibitem[MR03]{Re}
Colin Maclachlan and Alan~W. Reid.
\newblock {\em The arithmetic of hyperbolic 3-manifolds}, volume 219 of {\em
  Graduate Texts in Mathematics}.
\newblock Springer-Verlag, New York, 2003.

\bibitem[Mum07]{Mu}
David Mumford.
\newblock {\em Tata lectures on theta. {I}}.
\newblock Modern Birkh\"auser Classics. Birkh\"auser Boston Inc., Boston, MA,
  2007.
\newblock With the collaboration of C. Musili, M. Nori, E. Previato and M.
  Stillman, Reprint of the 1983 edition.

\bibitem[NN81]{NN}
M.~S. Narasimhan and M.~V. Nori.
\newblock Polarisations on an abelian variety.
\newblock {\em Proc. Indian Acad. Sci. Math. Sci.}, 90(2):125--128, 1981.

\bibitem[Otr71]{Ot}
Gertrud Otremba.
\newblock Zur {T}heorie der hermiteschen {F}ormen in imagin\"ar-quadratischen
  {Z}ahlk\"orpern.
\newblock {\em J. Reine Angew. Math.}, 249:1--19, 1971.

\bibitem[Pac05]{Pa}
Ariel Pacetti.
\newblock A formula for the central value of certain {H}ecke {$L$}-functions.
\newblock {\em J. Number Theory}, 113(2):339--379, 2005.

\bibitem[{PAR}08]{Pari}
{PARI~Group}, Bordeaux.
\newblock {\em {PARI/GP, version {\tt 2.3.4}}}, 2008.
\newblock available from http://pari.math.u-bordeaux.fr.

\bibitem[Piz80]{Pi}
Arnold Pizer.
\newblock An algorithm for computing modular forms on {$\Gamma \sb{0}(N)$}.
\newblock {\em J. Algebra}, 64(2):340--390, 1980.

\bibitem[Roh80a]{Ro1}
David~E. Rohrlich.
\newblock The nonvanishing of certain {H}ecke {$L$}-functions at the center of
  the critical strip.
\newblock {\em Duke Math. J.}, 47(1):223--232, 1980.

\bibitem[Roh80b]{Ro2}
David~E. Rohrlich.
\newblock On the {$L$}-functions of canonical {H}ecke characters of imaginary
  quadratic fields.
\newblock {\em Duke Math. J.}, 47(3):547--557, 1980.

\bibitem[Roh82]{Ro3}
David~E. Rohrlich.
\newblock Root numbers of {H}ecke {$L$}-functions of {CM} fields.
\newblock {\em Amer. J. Math.}, 104(3):517--543, 1982.

\bibitem[RV91]{V2}
Fernando Rodriguez~Villegas.
\newblock On the square root of special values of certain {$L$}-series.
\newblock {\em Invent. Math.}, 106(3):549--573, 1991.

\bibitem[RV93]{V3}
Fernando Rodriguez~Villegas.
\newblock Square root formulas for central values of {H}ecke {$L$}-series.
  {II}.
\newblock {\em Duke Math. J.}, 72(2):431--440, 1993.

\bibitem[RVZ93]{VZ}
Fernando Rodriguez~Villegas and Don Zagier.
\newblock Square roots of central values of {H}ecke {$L$}-series.
\newblock pages 81--99, 1993.

\bibitem[Shi64]{ShA}
Goro Shimura.
\newblock Arithmetic of unitary groups.
\newblock {\em Ann. of Math. (2)}, 79:369--409, 1964.

\bibitem[Shi71]{ShB}
Goro Shimura.
\newblock On elliptic curves with complex multiplication as factors of the
  {J}acobians of modular function fields.
\newblock {\em Nagoya Math. J.}, 43:199--208, 1971.

\bibitem[Shi73a]{Sh}
G.~Shimura.
\newblock On modular forms of half integral weight.
\newblock {\em Ann. of Math. (2)}, 97:440--481, 1973.

\bibitem[Shi73b]{ShC}
Goro Shimura.
\newblock On the factors of the jacobian variety of a modular function field.
\newblock {\em J. Math. Soc. Japan}, 25:523--544, 1973.

\bibitem[Sta82]{St}
H.~M. Stark.
\newblock On the transformation formula for the symplectic theta function and
  applications.
\newblock {\em J. Fac. Sci. Univ. Tokyo Sect. IA Math.}, 29(1):1--12, 1982.

\bibitem[Vig80]{Vig}
Marie-France Vign{\'e}ras.
\newblock {\em Arithm\'etique des alg\`ebres de quaternions}, volume 800 of
  {\em Lecture Notes in Mathematics}.
\newblock Springer, Berlin, 1980.

\bibitem[Wal80]{Wa1}
J.-L. Waldspurger.
\newblock Correspondance de {S}himura.
\newblock {\em J. Math. Pures Appl. (9)}, 59(1):1--132, 1980.

\bibitem[Wal81]{Wa}
J.-L. Waldspurger.
\newblock Sur les coefficients de {F}ourier des formes modulaires de poids
  demi-entier.
\newblock {\em J. Math. Pures Appl. (9)}, 60(4):375--484, 1981.

\bibitem[Zag02]{Za}
Don Zagier.
\newblock Traces of singular moduli.
\newblock In {\em Motives, polylogarithms and {H}odge theory, {P}art {I}
  ({I}rvine, {CA}, 1998)}, volume~3 of {\em Int. Press Lect. Ser.}, pages
  211--244. Int. Press, Somerville, MA, 2002.

\end{thebibliography}

\end{document}